\documentclass{amsart}
\usepackage{amsmath,amsthm,amscd,amssymb}
\usepackage{fancyhdr}
\def\a{\alpha}

\def\d{\delta}

\def\e{\epsilon}
\def\f{\frac}                          
\def\g{\gamma}
\def\G{\Gamma}
\def\k{\kappa}
\def\lb\{{\left\{}                     
\def\la{\lambda}
\def\La{\Lambda}
\def\lla{\longleftarrow}               
\def\lm{\limits}                       
\def\lra{\longrightarrow}              
\def\dllra{\Longleftrightarrow}        
\def\llra{\longleftrightarrow}         
\def\n{\nabla}
\def\ngth{\negthickspace}              
\def\ngtn{\negthinspace}              
\def\ola{\overleftarrow}               
\def\Om{\Omega}
\def\om{\omega}
\def\op{\oplus}
\def\oper{\operatorname}             
\def\oplm{\operatornamewithlimits}   
\def\ora{\overrightarrow}            
\def\ov{\overline}                   
\def\ova{\overarrow}                 
\def\ox{\otimes}                     
\def\p{\partial}                     
\def\rb\}{\right\}}                  
\def\s{\sigma}
\def\sbq{\subseteq}                  
\def\spq{\supseteq}                  
\def\sqp{\sqsupset}                  
\def\supth{{\text{th}}}              
\def\T{\Theta}
\def\th{\theta}
\def\tkap{\thickapprox}              
\def\tl{\tilde}
\def\tril{\triangleleft}             
\def\thra{\twoheadrightarrow}        
\def\un{\underline}                  
\def\ups{\upsilon}
\def\vp{\varphi}
\def\vt{\vartheta}
\def\wh{\widehat}                    
\def\wt{\widetilde}                  
\def\x{\times}                       
\def\z{\zeta}
\def\({\left(}
\def\){\right)}
\def\[{\left[}
\def\]{\right]}
\def\<{\left<}
\def\>{\right>}

\newcommand{\lng}{\langle}        
\newcommand{\rng}{\rangle}        
\newcommand{\bk}[1]{\langle #1\rangle}

\def\ra{\rightarrow}
\newcommand{\rz}{\raisebox{.2ex}{*}}
\newcommand{\rzc}{\raisebox{.1ex}{\circ}}
\newcommand{\rzf}{\raisebox{.1ex}{f}}
\newcommand{\thom}{\widetilde{\hom}}

\def\tec{Teichm\"uller\ }
\def\sconr{\hbox{\medspace\vrule width 0.4pt height 4.7pt depth
0.4pt \vrule width 5pt height 0pt depth 0.4pt\medspace}}
\def\SA{\mathcal A}
\def\SB{\mathcal B}
\def\SC{\mathcal C}
\def\SD{\mathcal D}
\def\SE{\mathcal E}
\def\SF{\mathcal F}
\def\SG{\mathcal G}
\def\SH{\mathcal H}
\def\SI{\mathcal I}
\def\SJ{\mathcal J}
\def\SK{\mathcal K}
\def\SL{\mathcal L}
\def\SM{\mathcal M}
\def\SN{\mathcal N}
\def\SO{\mathcal O}
\def\SP{\mathcal P}
\def\SQ{\mathcal Q}
\def\SR{\mathcal R}
\def\SS{\mathcal S}
\def\ST{\mathcal T}
\def\SU{\mathcal U}
\def\SV{\mathcal V}
\def\SW{\mathcal W}
\def\SX{\mathcal X}
\def\SY{\mathcal Y}
\def\SZ{\mathcal Z}

\newcommand{\BA}{\ensuremath{\mathbf A}}
\newcommand{\BB}{\ensuremath{\mathbf B}}
\newcommand{\BC}{\ensuremath{\mathbf C}}
\newcommand{\BD}{\ensuremath{\mathbf D}}
\newcommand{\BE}{\ensuremath{\mathbf E}}
\newcommand{\BF}{\ensuremath{\mathbf F}}
\newcommand{\BG}{\ensuremath{\mathbf G}}
\newcommand{\BH}{\ensuremath{\mathbf H}}
\newcommand{\BI}{\ensuremath{\mathbf I}}
\newcommand{\BJ}{\ensuremath{\mathbf J}}
\newcommand{\BK}{\ensuremath{\mathbf K}}
\newcommand{\BL}{\ensuremath{\mathbf L}}
\newcommand{\BM}{\ensuremath{\mathbf M}}
\newcommand{\BN}{\ensuremath{\mathbf N}}
\newcommand{\BO}{\ensuremath{\mathbf O}}
\newcommand{\BP}{\ensuremath{\mathbf P}}
\newcommand{\BQ}{\ensuremath{\mathbf Q}}
\newcommand{\BR}{\ensuremath{\mathbf R}}
\newcommand{\BS}{\ensuremath{\mathbf S}}
\newcommand{\BT}{\ensuremath{\mathbf T}}
\newcommand{\BU}{\ensuremath{\mathbf U}}
\newcommand{\BV}{\ensuremath{\mathbf V}}
\newcommand{\BW}{\ensuremath{\mathbf W}}
\newcommand{\BX}{\ensuremath{\mathbf X}}
\newcommand{\BY}{\ensuremath{\mathbf Y}}
\newcommand{\BZ}{\ensuremath{\mathbf Z}}

\def\Ff{\mathfrak f}
\def\Fg{\mathfrak g}

\def\Fr{\mathfrak r}
\def\Fs{\mathfrak s}
\def\Ft{\mathfrak t}

\def\bba{{\mathbb A}}
\def\bbb{{\mathbb B}}
\def\bbc{{\mathbb C}}
\def\bbd{{\mathbb D}}
\def\bbe{{\mathbb E}}
\def\bbf{{\mathbb F}}
\def\bbg{{\mathbb G}}
\def\bbh{{\mathbb H}}
\def\bbi{{\mathbb I}}
\def\bbj{{\mathbb J}}
\def\bbk{{\mathbb K}}
\def\bbl{{\mathbb L}}
\def\bbm{{\mathbb M}}
\def\bbn{{\mathbb N}}
\def\bbo{{\mathbb O}}
\def\bbp{{\mathbb P}}
\def\bbq{{\mathbb Q}}
\def\bbr{{\mathbb R}}
\def\bbs{{\mathbb S}}
\def\bbt{{\mathbb T}}
\def\bbu{{\mathbb U}}
\def\bbv{{\mathbb V}}
\def\bbw{{\mathbb W}}
\def\bbx{{\mathbb X}}
\def\bby{{\mathbb Y}}
\def\bbz{{\mathbb Z}}

\def\rns{\rz\bbr^n_{nes}}            
\def\rms{\rz\bbr^m_{nes}}            
\def\rps{\rz\bbr^m_{nes}}            
\def\rs{\rz\bbr_{nes}}            

\newcommand{\bsm}[1]{\boldsymbol{#1}}    

\def\ssm{\smallsetminus}

\def\fst{\frak{st}}

\newcommand{\cMn}{$C^\infty(M,\bbr^n)$}
\newcommand{\dcMn}{C^\infty(M,\bbr^n)}
\newcommand{\cM}{$C^\infty(M,\bbr )$}
\newcommand{\dcM}{C^\infty(M,\bbr )}
\newcommand{\cmn}{$C^\infty(\bbr^m,\bbr^n)$}
\newcommand{\cnn}{$C^\infty(\bbr^n,\bbr^n)$}
\newcommand{\dcmn}{C^\infty(\bbr^m,\bbr^n)}
\newcommand{\dcnn}{C^\infty(\bbr^n,\bbr^n)}
\newcommand{\cm}{$C^\infty(\bbr^m,\bbr)$}
\newcommand{\dcm}{C^\infty(\bbr^m,\bbr)}
\newcommand{\cn}{$C^\infty(\bbr^n,\bbr)$}
\newcommand{\dcn}{C^\infty(\bbr^n,\bbr)}
\newcommand{\cnpr}{$C^\infty_{pr}(\bbr^n,\bbr)$}
\newcommand{\cnnpr}{$C^\infty_{pr}(\bbr^n,\bbr^n)$}
\newcommand{\dcnpr}{C^\infty_{pr}(\bbr^n,\bbr)}
\newcommand{\dcnnpr}{C^\infty_{pr}(\bbr^n,\bbr^n)}

\newcommand{\cjkmn}{C^{\infty}(\SJ^k_{m,n})}
\newcommand{\cjkmom}{C^{\infty}(\SJ^k_{m,1},\bbr_m)}

\newcommand{\strcmn}{$\rz C^\infty(\bbr^m,\bbr^n)$}
\newcommand{\dstrcmn}{\rz C^\infty(\bbr^m,\bbr^n)}
\newcommand{\strcm}{$\rz C^\infty(\bbr^m,\bbr)$}
\newcommand{\dstrcm}{\rz C^\infty(\bbr^m,\bbr)}
\newcommand{\strcn}{$\rz C^\infty(\bbr^n,\bbr)$}
\newcommand{\strcnn}{$\rz C^\infty(\bbr^n,\bbr^n)$}
\newcommand{\dstrcn}{\rz C^\infty(\bbr^n,\bbr)}
\newcommand{\dstrcnn}{\rz C^\infty(\bbr^n,\bbr^n)}
\newcommand{\strcnpr}{$\rz C^\infty_{pr}(\bbr^n,\bbr)$}
\newcommand{\dstrcnpr}{\rz C^\infty_{pr}(\bbr^n,\bbr)}
\newcommand{\strcMn}{$\rz C^\infty(M,\bbr^n)$}
\newcommand{\dstrcMn}{\rz C^\infty(M,\bbr^n)}
\newcommand{\strcM}{$\rz C^\infty(M,\bbr)$}
\newcommand{\dstrcM}{\rz C^\infty(M,\bbr)}

\newcommand{\scmn}{$SC^\infty(\bbr^m,\bbr^n)$}
\newcommand{\scnn}{$SC^\infty(\bbr^n,\bbr^n)$}
\newcommand{\scn}{$SC^\infty(\bbr^n,\bbr)$}
\newcommand{\scnnpr}{$SC^\infty_{pr}(\bbr^n,\bbr^n)$}
\newcommand{\dscmn}{SC^\infty(\bbr^m,\bbr^n)}
\newcommand{\dscnn}{SC^\infty(\bbr^n,\bbr^n)}
\newcommand{\dscnnpr}{SC^\infty_{pr}(\bbr^n,\bbr^n)}
\newcommand{\scm}{$SC^\infty(\bbr^m,\bbr)$}
\newcommand{\dscm}{SC^\infty(\bbr^m,\bbr)}
\newcommand{\dscn}{SC^\infty(\bbr^n,\bbr)}
\newcommand{\scMn}{$SC^\infty(M,\bbr^n)$}
\newcommand{\dscMn}{SC^\infty(M,\bbr^n)}
\newcommand{\scM}{$SC^\infty(M,\bbr)$}
\newcommand{\dscM}{SC^\infty(M,\bbr)}

\newcommand{\cgcmn}{$^\s\! C^\infty(\bbr^m,\bbr^n)$}
\newcommand{\cgcm}{$^\s\! C^\infty(\bbr^m,\bbr)$}
\newcommand{\cgcMn}{$^\s\! C^\infty(M,\bbr^n)$}
\newcommand{\tangm}{$C^\infty(T\bbr^m)$}
\newcommand{\tangM}{$C^\infty(TM)$}
\newcommand{\stangM}{$SC^\infty(TM)$}
\newcommand{\tangMnes}{$C^\infty(TM)_{\nes}$}
\newcommand{\strtangm}{$\rz C^\infty(T\bbr^m)$}
\newcommand{\strtangM}{$\rz C^\infty(TM)$}
\newcommand{\tangn}{$C^\infty(T\bbr^n)$}
\newcommand{\strtangn}{$\rz C^\infty(T\bbr^n)$}
\newcommand{\tanf}[1]{$C^\infty(#1^{-1}T\bbr^n)$}
\newcommand{\tanfnes}[1]{$C^\infty(#1^{-1}T\bbr^n)_{nes}$}
\newcommand{\strtanf}[1]{$\rz C^\infty(#1^{-1}T\bbr^n)$}
\newcommand{\dcgcmn}{^\s \!C^\infty(\bbr^m,\bbr^n)}
\newcommand{\dcgcm}{^\s\! C^\infty(\bbr^m,\bbr)}
\newcommand{\dcgcMn}{^\s\! C^\infty(M,\bbr^n)}
\newcommand{\dtangm}{C^\infty(T\bbr^m)}
\newcommand{\dtangM}{C^\infty(TM)}
\newcommand{\dstangM}{SC^\infty(TM)}
\newcommand{\dtangMnes}{C^\infty(TM)_{\nes}}
\newcommand{\dstrtangm}{\rz C^\infty(T\bbr^m)}
\newcommand{\dstrtangM}{\rz C^\infty(TM)}
\newcommand{\dtangn}{C^\infty(T\bbr^n)}
\newcommand{\dstrtangn}{\rz C^\infty(T\bbr^n)}
\newcommand{\dtanf}[1]{C^\infty(#1^{-1}T\bbr^n)}
\newcommand{\dtanfnes}[1]{C^\infty(#1^{-1}T\bbr^n)_{nes}}
\newcommand{\dstrtanf}[1]{\rz C^\infty(#1^{-1}T\bbr^n)}

\newcommand{\cts}[1]{\*C^\infty(T\bbr^#1)}

\newcommand{\RDM}{\SD_M^{\bbr}}
\newcommand{\lRDn}{\SD_n^{\bbr}}

\newcommand{\DM}{\SD_{M}}
\newcommand{\DMs}{\SD_{M,nes}}
\newcommand{\Dn}{\SD_{n}}
\newcommand{\Dns}{\SD_{n,nes}}
\newcommand{\flXt}{$\BF l^X_t$}           
\newcommand{\flXs}{$\BF l^X_s$}
\newcommand{\flYt}{$\BF l^Y_t$}
\newcommand{\flYs}{$\BF l^Y_s$}
\newcommand{\flXd}{$\BF l^X_{\d}$}
\newcommand{\flYd}{$\BF l^Y_{\d}$}
\newcommand{\dflXt}{\BF l^X_t}           
\newcommand{\dflXs}{\BF l^X_s}
\newcommand{\dflYt}{\BF l^Y_t}
\newcommand{\dflYs}{\BF l^Y_s}
\newcommand{\dflXd}{\BF l^X_{\d}}
\newcommand{\dflYd}{\BF l^Y_{\d}}

\newcommand{\norm}[1]{\left\|#1\right\|}
\newcommand{\normk}[1]{\left\|#1\right\|_k}
\newcommand{\normkk}[1]{\left\|#1\right\|_{k+1}}
\newcommand{\normkpl}[1]{\left\|#1\right\|_{k+1}}
\newcommand{\normo}[1]{\left\|#1\right\|_1}
\newcommand{\normz}[1]{\left\|#1\right\|_0}
\newcommand{\normj}[1]{\left\|#1\right\|_j}
\newcommand{\nsnormj}[1]{\rz \left\|#1\right\|_j}

\newcommand{\quo}[1]{\frac{\norm{#1}}{\norm{#1}+1}}
\newcommand{\quok}[1]{\frac{\normk#1}{\normk#1+1}}
\newcommand{\quoj}[1]{\frac{\normj#1}{\normj#1+1}}
\newcommand{\quoke}[1]{\frac{\normk#1}{\normkpl#1}}
\newcommand{\quoje}[1]{\frac{\normj#1}{\normj#1}}
\newcommand{\nsquoj}[1]{\frac{\rz \ngtn\normj#1}{*\ngtn\normj#1+1}}
\newcommand{\smk}[1]{\sum_{k=0}^{\infty}\frac{1}{2^k}#1}
\newcommand{\smj}[1]{\sum_{j=0}^{\infty}\frac{1}{2^j}#1}
\newcommand{\nssmj}[1]{\rz \ngtn\sum_{j=0}^{\infty}\frac{1}{2^j}#1}

\newcommand{\pow}[2]{\ensuremath{$#1^#2$}}

\def\k{\kappa}
\usepackage{hyperref}

\newcommand{\btrn}[1]{\bigtriangledown_#1}
\newcommand{\btrndg}{\bigtriangledown_{\d}(g)}
\newcommand{\trndg}{\raisebox{.1ex}{\bigtriangledown_{\d}(g)}}
\newcommand{\fpd}{\rz f\circ \varphi_{\d}}
\newcommand{\frad}{\frac{1}{\d}}
\newcommand{\pdf}{\psi_{\d} \circ \rz f}

\newcommand{\dsim}{\stackrel{\d}{\sim}}

\newtheorem{theorem}{Theorem}[section]
\newtheorem{lemma}{Lemma}[section]

\newtheorem{definition}{Definition}[section]
\newtheorem{corollary}{Corollary}[section]
\newtheorem*{notation}{Notation}
\theoremstyle{definition}
\newtheorem*{remark}{Remark}

\title{Generalized Solutions of Nonlinear Differential Equations \\ A Nonstandard Jets Approach
  }
\author{Tom McGaffey}

\begin{document}

\begin{abstract}
    Using the rudiments of pde jets theory in a nonstandard setting, we first deepen and extend previous nonstandard existence results for generalized solutions of linear differential equations and second extend the previous results for linear differential equations to a much broader class of nonlinear differential equations.
\end{abstract}

\maketitle

\tableofcontents

\begin{abstract} We extend a recently proved result of Todorov
asserting the existence of generalized solutions of very general
linear partial differential operators. To linear operators whose
symbols vanish only to finite order, we prove existence of solutions
to infinite order on $^\s\bbr^m$. We prove existence of generalized
solutions for nonlinear operators satisfying $^\s\! PCP$, a
condition that implies Todorov's condition in the linear case. In
the conclusion, we prove that generalized solutions on $^\s\bbr^m$
are remarkably abundant .

\end{abstract}









\section{Introduction} In this paper, we extend the  results
   of Todorov \cite{Todorov96} (on the existence of generalized solutions for a general set of differential operators) in two directions.
   If $P$ is a linear partial differential operator of order r, written as $P\in LPDO(r)$, with
   $C^{\infty}$ coefficients, and ${\la}_P$ is its (total) symbol, then Todorov
   proves the existence of generalized solutions $f$ for the equation $\rz P(f)(\rz x)=\rz g(\rz x)$ for all $x\in\bbr^m$ outside $\SZ_{{\la}_P}$  and for quite
   general $g$, where ${\SZ}_h$ denotes the
   $\{x\in\bbr^n:h(x)=0\}$. From a slightly different perspective, Todorov's result says that for a general set of standard $g$,
   there exists internal $f\in\rz\dcm$, such that $(\rz\! P(f)-\rz\! g)|_{^\s\bbr^m}=0$.
   ($^\s\bbr^m$ denotes the standard vectors in the internal vector space $\rz\bbr^m$).
   That is, $\rz P(f)-g$ vanishes pointwise, ie., has $0^{th}$ order contact with $\rz\bbr^n$ at each point of $^\s\bbr^m$. The standard geometry and jet definitions with respect to PDEs (partial differential equations) are recalled in the next section. The first extension of Todorov in this paper is to give a straightforward construction that there exists
   internal smooth maps $f$ such that  $\rz\! P(f)-\rz\! g$ has
   infinite order contact with $\rz\bbr^n$ at all points of $^\s\bbr^m$. (See Theorem \ref{thm:lin eqn,infinite contact
   soln}.) Another way of stating this is that Todorov solves the equations at each standard $x$; here we solve the equations along with the infinite family of integrability differential equations associated to the given differential equation at each such $x$. But the two corollaries carry the critical import of this
   theorem. Corollary \ref{cor: infinite order Todorov} is the
   infinite order direct descendant of Todorov's result. It depends
   on the standard jet space work done in Section \ref{section: standard jet work,prolong and
   rank}. The needed standard statement following from this work  lies in Corollary \ref{lem:symb prolong is
   surj}. The second corollary becomes possible only within the
   perspective of this paper. We can consider those partial differential operators
   whose (total) symbols vanish to some finite order, ie., have any finite order contact with $\rz\bbr^m$ along $^\s\bbr^m$; see Section \ref{section:prolong and
   vanishing}\;. Note here that Todorov only consider the $0^{th}$ order vanishing case. Our Corollary \ref{cor: infinite order solns for finite
   contact} says that as long as the vanishing order of $\rz g$ at
   standard points is controlled by that of the symbol of $\rz P$, we
   can find internal smooth $f$ solving $\rz P(f)-\rz g =0$ to
   infinite order on $^\s\bbr^m$. Such a theorem is unstatable
   within the venue of Todorov's setting. The work in standard
   geometry allowing the proof of this result occurs in section \ref{section:prolong and
   vanishing}; see Corollary \ref{cor:removing finite zeroes}. Parenthetically, it's conceivable that the PDE jet results of Sections \ref{section:prolong and vanishing} and \ref{section: standard jet work,prolong and rank} exist in the literature, but the author could not find them.

   The overwhelming bulk of the work in this paper concerns the linear theory. But, the nonlinear PDE jet framework and NSA fit quite well together, and so the second direction of extension of the result of Todorov is into nonlinear partial differential equations, NLPDEs.  There is
   a well developed theory of nonlinear partial differential equations within the jet bundle framework,
   exemplified in the texts of Pommaret,\cite{Pommaret1978}
  Olver, \cite{Olver1993} and Vinogradov, \cite{VinogradovGeomJetSpNLDE1986}. Nonstandard analysis is as comfortable in this framework as in the linear.
   So, in Section \ref{section: nonlinear work}, we introduce  simple conditions, $PCP$, and
   ${} ^\s PCP$, on the symbols of general NLPDEs of finite order, and give an easy proof of
   existence of generalized solutions in the sense of Todorov for
   those NPDO's satisfying these criteria. We show that Todorov's
   nonvanishing condition on the symbol implies that his $LPDO$'s
   satisfy ${}^\s PCP$. But, our theorem asserts the existence of generalized solutions in the
   far broader nonlinear arena.

    The standard import of these results is yet to be
   worked out. See the conclusion for a curious result on this. We prove a result
   that might appear startling: almost all internal smooth functions are solutions on $^\s\bbr^m$ of any standard differential
   operator that has the zero function as a solution.
   It seems that the work of Baty, etal., \cite{BatyOneDimlGas2007}, might be a useful framing for this. That is, their analysis needs a lot of elbow room on the infinitesimal level to allow adjusting eg., the infinitesimal widths of Heaviside jumps, etc. It seems that the results here might be interpreted as saying that the formal (nonstandard) jet theory of symbols allows such roominess for such empirically motivated adjustments.

   The author relies on a jet bundle framework when some might consider it too
   big a machine for the job.  Yet, from the point of view of
   nonstandard analysis, the jet bundle framework is natural and eg., allows an easy
   generalization of Todorov's result to the nonlinear case. The total and principal symbols  of a differential operator have a natural
   geometric setting which when extended to the nonstandard world
   allows a geometric consideration of generalized solutions and, in fact generalized differential operators vis *smooth symbols.

   Todorov defines his differential equation and  constructs his
   solutions within spaces of generalized functions defined on $\rz\Om$, where $\Om$ an arbitrary
   open subset of $\bbr^m$, and gets his localizable differential
   algebra of generalized functions by `quotienting' out by the
   parts of the $\rz C^{\infty}$ functions defined on
   nonnearstandard points of $\rz \Om$. (Note also his NSA jazzed up version of the constructions of Colombeau,  Oberguggenberger, and company in eg., \cite{TodorVern2008}, of which we will say more later).
   On the other hand, my paper focuses almost exclusively
   on extending Todorov's existence result to more general classes
   of differential operators and very little on a broader analysis
   of his differential algebra of generalized functions. (In a follow up to this paper, we will refine the results appearing here within the aforementioned  nonstandard version of Colombeau's algebra of generalized functions constructed by Todovov and Oberguggenberger, \cite{OT98}).
   Accordingly, our constructions occur on all of $\rz\bbr^m_{nes}$.
   If we restrict to differential operators whose
     finite vanishing order sets don't have infinitely many components
     so that they have no
   nontrivial *limiting behavior at nonnearstandard points,
   the  results here should hold without change for Todorov's
   localizable   differential algebras.

   The
   geometric theory of differential equations and their symmetries,
   as exemplified in eg., Olver, \cite{Olver1993} and Pommaret, \cite{Pommaret1978}.
   is a natural framework within which to integrate the
   generalizing notions of NSA. This is the first of a series of
   papers in which the author intends to attempt a theory of generalized solutions (existence and regularity) and symmetries of
   differential equations within the context of the extensive jet theory.
   Note that although this approach seems to be new,
   there are a growing number of research programs moving beyond classical approaches; see eg., Colombeau, \cite{Colombeau1985},  Oberguggenberger, \cite{Oberguggengerger1992}
   and Rossinger, \cite{Rosinger1987}. For a good overview of the new theories of generalized functions, see eg., Hoskins and Pinto, \cite{HoskinsPintoGeneralizFcns2005}, and for specific  surveys of the obstacles to the construction of
   a nonlinear generalization of distributions and a comparison of
   the characteristics of the these new theories, see
   Oberguggenberger, \cite{Oberguggenberger1995a} and more recently Colombeau, \cite{ColombeauSurvGeneralizFcnsInfinites2006}.
    Note , in particular the flury of work extending the arena of Colombeau algebras into mathematical physics involving diferential geometry and topology that Kunzinger, \cite{Kunzinger2007}, summarizes.

   Further note that Oberguggenberger and
   Todorov, see eg., \cite{OT98}, have shown how much of the
   theoretical foundations of Schwarz type Colombeau algebras can be simplified
   and strengthened within the venue of nonstandard constructions and  recent work of Todorov and coworkers, see eg., \cite{TodorVern2008}, have extended the results with this model. See also Todorov's lecture notes, \cite{TodorovNotes}.
   None of these approaches consider the symbol of the
   differential operator, and the nonstandard extension of its geometric milleu as the primary
   object of study. This is the perspective of the current work. Finally, it seems that nonstandard methods are much more encompassing than the impressive work of the Colombeau school of generalized functions; eg., consider the work of the mathematical physicists working around Baty, eg., see \cite{BatyOneDimlGas2007} and \cite{BatyShockwaveStar2009}. Note, in particular the perspective of Baty, etal on p37 of \cite{BatyOneDimlGas2007} (with respect to the benefits of nonstandard methods)  where they note that the generalized functions of the Columbeau school

\begin{quote}
          ...are not smooth functions and do not support all of the operations of ordinary algebra and calculus, the multiplication of singular generalized functions is accomplished via a weak equality called association. In contrast to such calculations, the objects manipulated in equations (4.4) to (4.13) (and indeed in the following section of this report) are smooth nonstandard functions.
\end{quote}
       In reading their papers, it's clear that their need for ``ordinary algebra'', etc., is critical to their analysis. The author also believes that, here also, having at hand the full capacity of mathematics via transfer straightforwardly allows many of the  constructions of this paper.

\section{Some jet PDE basics and nonstandard variations}

\subsection{Nonstandard analysis}
\subsubsection{Resources}
    Good introductions to nonstandard analysis abound. One might start with the pedestrian tour of its basics in the introduction of the authors dissertation, \cite{McGaffeyPhD}; then get  deeper with the introduction of Lindstr{\o}m, \cite{Lindstrom1988} and follow this with the constructive introduction of Henson, \cite{Henson1997}. There is also Nelson's (axiomatic) internal set theory, a theory with similar goals and achievements to Robinson's (currently superstructure/ultrapower) nonstandard analysis; a good text being that of Lutz and Goze, \cite{LutzGoze1981}. One could also check out strategic outgrowths from these nonstandard schools, eg., the work of Di Nasso and Forti, \cite{DiNassoFortiTopExtens2006}. One might also check out their (and Benci's) constructive survey, \cite{BenciFortiNassoEightfoldPath2006}, of a variety of means to a nonstandard mathematics (which also includes a good introduction). There are yet other approaches to a nonstandard mathematics, notably the Russian  school; for a good example, see \cite{InfinitesAnalyGordonKusraevKutateladzeBk2002}.
\subsubsection{Impressionistic introduction with terminology}
     Let's give an impressionistic introduction to nonstandard mathematics via the (extended) ultrapower constructions.
     There are many motivations for the need of a nonstandard mathematics. To have our real numbers, $\bbr$, embedded in a much more robust object, $\rz\bbr$, with the properties of the real numbers, but also containing infinite and infinitesimal quantities is a boon to a direct formalization of intuitive strategies. Model theoretically, these have been around for more than 60 years (some would argue much longer) via  eg., ultrapowers or the compactness theorem. 
     
     The ultrapower  is the construction generally least involved with theoretical matters of the foundations of math (but see \cite{DiNassoFortiTopExtens2006}). Thinking of eg., infinite numbers as limiting properties of sequences of real numbers, one might attempt to construct nonstandard real numbers as equivalence classes of such sequences, ie., $\rz\bbr=\bbr^\bbn/\sim$ where $/\!\sim$ denotes the forming of such equivalence classes. Clearly, one can extend the operations and relations on $\bbr$ to $\bbr^\bbn$ coordinate wise, getting a partially ordered ring; but almost all of the nice properties of $\bbr$ are lost.
     But, if $\SP(\bbr)=2^\bbn$, the power set of $\bbn$,  it turns out that there are objects $\SU\subset\SP(\bbr)$, the ultrafilters, such that defining our equivalence relation in terms of elements of $\SU$ preserves all ``well stated'' properties of $\bbr$. More specifically, given $(r_i),(s_i)\in\bbr^\bbn$, define $(r_i)\sim (s_i)$ if  $\{i:r_i=s_i\}\in\SU$ and  $(r_i)<(s_i)$ if, again $\{i:r_i<s_i\}\in\SU$, we find that the extended ring operations and partial order are well defined on the quotient $\bbr^\bbn/\SU$ and, in fact, it's not hard to prove that we get a totally ordered field containing $\bbr$  (the set of equivalence classes of constant sequences) as a subfield.
     For $r\in\bbr$, let $\rz r=(r)/\sim$, the equivalence class containing the corresponding constant sequence and $\bsm{{}^\s A}=\{\rz r:r\in A\}\subset\rz\bbr$ denote the image of $A\subset\bbr$ in our nonstandard model of $\bbr$; eg., ${}^\s\bbr$ is the image of $\bbr$. In general, we will let $\bsm{\bk{r_i}}$ denote the equivalence class of a sequence $(r_i)\in\bbr^\bbn$.  
     Given this, existence of infinite elements is clear:  if $(r_i)\in\bbr^\bbn$ with $r_i\ra\infty$ as $i\ra\infty$, then for each $s\in\bbr$, the set $\{i:r_i>s\}$ is in $\SF(\subset\SU)$ and so, by our definition, $\bk{r_i}>\rz s$. Note that to verify the field properties and total ordering of $\rz\bbr$, we need the full strength of ultrafilters, eg., the maximality property: if $A\subset\bbn$, then precisely one of $A$ or $\bbn\ssm A$ is in $\SU$.
     In particular, if $\om=\bk{m_i}\in\rz\bbn$ where $m_i\uparrow\infty$ as $i\ra\infty$, eg., $\om$ is infinite

     Since we can form the $\SU$ equivalence class of arbitrary sequences of real numbers and get a much larger field with all of the `same' well formed properties as $\bbr$, why can't we do this for $\bbn$, $\bbq$, $\bbq(\sqrt{5})$, etc. and get `enlarged' versions of these? We can, but try to do this with the algebra $F(\bbr)$ of real valued smooth maps on $\bbr$; ie., consider $F(\bbr)^\bbn/\sim$ as before. Clearly, this is a ring, but do these `functions' (on $\rz\bbr$) have, in some good sense, all of the properties of functions in $F(\bbr)$? Ignoring subtleties, the simple answer is yes, simply because these elements are internal and therefore fall under the aegis of the all encompassing {\it principle of transfer}; but let's see, to some extent, how this works in this case.
     Let's consider, for example, the nonstandard support (*support) of an equivalence class $\bk{f_i}\in\rz F(\bbr)$. (Recall that if $f\in F(\bbr)$, then the support of $f$, $supp(f)$, is the closure of the set of $t\in\bbr$ where $f(t)\not=0$.) But, then as we seem to be following a recipe of extending everything component wise and then taking the quotient, if $A_i=supp(f_i)$, then $\rz supp(\bk{f_i})$ must be the equivalence class $\bk{A_i}$.

     Yet, how is this a subset of $\rz\bbr$?
     This is a special case of the next problem of nonstandard analysis: extending `is an element of' to our ultrapower constructions. Miraculously, the properties of ultrafilters (eg., our $\SU$) allow one to (simplemindedly!) define $\bk{r_i}\in\bk{A_i}$ if $\{i:r_i\in A_i\}\in\SU$. (This really should be written $\bk{r_i}\;\rz\!\!\in\bk{A_i}$, but starring all extended operations, relations, etc. can rapidly get confusing.) Note that these subsets of $\SP(\rz\bbr)$ of the form $\bk{A_i}$ are called {\it internal sets} and are {\it precisely those subsets that extend the properties of $\SP(\bbr)$ (and therefore shall be denoted $\rz\SP(\bbr)$) via the principle of transfer}.
     For example, the typical bounded subset $\SC$ of $\rz\bbr$ does not have a nonstandard supremum, ie., $\rz\sup \SC$ does not exist; in particular, the transfer principle applied to the theorem that bounded subsets of $\bbr$ have suprema does not transfer to all *bounded elements of $\SP(\rz\bbr)$. (For example, the set of {\it infinitesimals}, denoted $\mu(0)$ here, certainly does not have a supremum.) Nonetheless, the transfer principle certainly does apply to the internal $\bk{A_i}$, and the proper definition is (surprise!)  $\rz\sup\bk{A_i}=\bk{\sup A_i}$. Note here that other notable examples of external subsets of $\rz\bbr$ are ${}^\s\bbr$ (and in fact ${}^\s A$ for any infinite subset $A\subset\bbr$), 
      \begin{align}     
         \rz\bbr_{nes}=\{\Ft\in\rz\bbr:|\Ft-\rz s|\in\mu(0)\;\text{for some}\;s\in\bbr\}, 
      \end{align}
     the {\it nearstandard real numbers}, $\rz\bbr_\infty$, the infinite real numbers and the infinite natural numbers, $\rz\bbn_\infty$ which are said to be {\it *finite}. If $A\subset\bbr$ is finite, let $|A|\in\bbn$ denote its cardinality, let $\om=\bk{m_i}\in\rz\bbn$ be an infinite *finite integer and $A_i\subset\bbr$ be such that $\{i:|A_i|=m_i\}\in\SU$. Then we say that $\bk{A_i}\subset\rz\bbr$ is a {\it *finite subset of *cardinality} $\om$. Although (for $\om$ infinte) these sets are infinite (in fact uncountable!), transfer implies that *finite subsets of $\bbr$ have the `same' properly stated properties that finite subsets have. Nonetheless, for a much stronger {\it sufficiently saturated} ultrapowers, there exists *finite subsets of $\rz\bbr$ containing ${}^\s\bbr$. These will play a role in this paper.

     We still haven't considered how the elements of $\rz F(\bbr)=F(\bbr)^\bbn/\sim$ can be considered as functions on $\rz\bbr$, but by now the reader can see that we must define $\bk{f_i}(\bk{x_i})=\bk{f_i(x_i)}$ and hope that the properties of $\SU$ ensure that this is well defined (ie., independent of choice of representatives) and is a function. This can indeed be verified and these functions are the {\it internal functions} in $F(\rz\bbr)$; the function $\Ff:\rz\bbr\ra\rz\bbr$ defined by $\Ff(x)=x$ if $x\sim 0$, ie., if $x$ is infinitesimal, and $\Ff(x)=0$ if $x\not\sim 0$ is an {\it external function}, eg., does not satisfy the internality criteria allowing the use of transfer. For example, it is *bounded (bounded in $\rz\bbr$), but $\rz\sup\Ff$ does not exist. Yet again, it's straightforward that for *bounded $\bk{f_i}$, $\rz\sup\bk{f_i}$ is well defined precisely by our recipe: $\bk{\sup f_i}\in\rz\bbr$ (this *supremum may be an infinite element of $\rz\bbr$). Internal subsets of $\rz\bbr$ of the form $\bk{A}$ (ie., the equivalence class containing the constant sequence $(A_i)$ for some $A\subset\bbr$) are called the {\it standard sets}. Following our recipe for denoting the equivalence class of a constant sequence by starring, $\bk{A}$ is usually denoted $\rz A$.
     For perspective, note that the copy of $[0,1]$ lying in $\rz[0,1]$, ie., ${}^\s[0,1]\subset\rz[0,1]$ is very sparse. For example, given an infinitesimal, $0<\Ft=\bk{t_i}\in\rz\bbr$ (eg., suppose  $t_i\downarrow 0$ as $i\ra\infty$) and $r\in(0,1)$, then $\rz r+[-\Ft,\Ft]\subset\rz[0,1]$, but intersects ${}^\s[0,1]$ only at $\rz r$. 
     Let's consider the {\it standard function} $\rz \sin(x)$. First of all, $\rz\sin$ is defined essentially as we defined standard sets, $\rz A$, ie., the $f_i$ above are all the function $\sin$. So if we define the {\it *domain} of $\bk{f_i}$ as we have all else: $\rz dom(\bk{f_i})\dot=\bk{dom(f_i)}$, we see that $\rz dom(\rz\sin)$ is all of $\rz\bbr$. (Or, as the domain of $\sin$ is $\bbr$, transfer says that the *domain of $\rz\sin$ is $\rz\bbr$.) A consequence of our constructive approach is the fact that $\rz dom(\rz\sin)=\bk{dom(f_i)}$ is internal. 
     It's not hard to check that $\rz\sin$ is really an extension of $\sin:\bbr\ra[-1,1]$: first, restricting the graph of $\rz\sin$ to ${}^\s\bbr$  just the image of the graph of $\sin$ in $\rz\bbr^2$; second, all of the symmetry and character properties hold and third, it has all of the (transferred) analytic properties that $\sin$ has.
     
     Before we conclude this tour, let's look at the standard part map. We defined the external (subring) $\rz\bbr_{nes}\subset\rz\bbr$ above. It should not be surprising that this is precisely those $\bk{r_i}\in\rz\bbr$ satisfying $|\bk{r_i}|<\rz t$ for some $t\in\bbr$ (here $|\;|:\rz\bbr\ra\rz[0,\infty)$ is defined as all else). But by it's definition, any $\bk{r_i}\in\rz\bbr_{nes}$ satisfies $\bk{r_i}\sim \rz r$ for some (clearly unique) $r\in\bbr$, eg., there is a well defined map (homomorphism onto!) $\fst:\rz\bbr_{nes}\ra\bbr$, the {\it standard part map}. Sometimes we will write ${}^o\bk{r_i}$ for $\fst\bk{r_i}$. 
     Note then that if $\bk{f_i}:\rz\bbr\ra\rz\bbr$ has image in $\rz\bbr_{nes}$, then we can define $\fst\bk{f_i}:\bbr\ra\bbr$ to be the map $r\in\bbr\mapsto\fst(\bk{f_i}(\rz r))$  
     Given this, if $\om=\bk{m_i}\in\rz 2\bbn$ with $m_i\uparrow\infty$ as $i\ra\infty$ (eg., $\om$ is infinite), consider $f_i$ given by $x\mapsto\sin(m_ix)$, so that writing $\xi=\bk{x_i}\in\rz\bbr$, we have $\rz\sin(\om \xi)=\bk{sin(m_ix_i)}$. By transfer, $\xi\mapsto\rz\sin(\om\xi)$ has all of the symmetry and analytic properties of $x\mapsto\sin(2mx)$ for some $m\in\bbn$, eg., solves the nonstandard *differential equation $\Ff''-\om^2\Ff=0$; yet it's standard part is not even Lebesgue measurable! 

\subsubsection{Formal tools}
    The four pillars of nonstandard analysis are the internal definition principle, transfer, saturation
     and (several versions of) ``overflow''. In order to discern the internal sets among all external sets, one can use the internal definition principle. It is basically an algorithmic way of determining if some object is of the form $\bk{S_i}$ and depends on the fact that all internal sets ar elements of some standard set $\rz T$. It asserts that if $\SB=\bk{B_i}$ is internal and $P(H)$ is a statement about an variable quantity $H$ in an internal set $\SX$ (of *functions, *measures, etc.,), then $\{H\in\SX:P(H)\;\text{is true}\}$ is internal.  
    As described above transfer allows us to, eg., translate to the
    nonstandard world careful statements about regular mathematics. Here we will need it to, eg., transfer to the nonstandard world the existence of maps of a certain type
    that have specified values on finite sets.
    Next, saturation has a variety of guises, one of which will be important here. Besides the need for the monads associated with neighborhood filters for a given topology, the specific form of saturation (see Stroyan and Luxemburg, \cite{StrLux76}, chapter 8) that will
    be needed here, in section 4, ensures that  *finite set are
    sufficiently large. Specifically, let  $X\in\SU$ be an infinite
    set of cardinality not bigger than $\SP(\dcm)$. Then, there
    exists a *finite $\SA\in\SU$ with ${}^\s X\subset\SA\subset\rz
    X$. This can be situated so that $\SA$ carries the same, well formed finitely stated, characteristics as $X$ (transfer). We will use this in the situation where $X$ is a particular
    collection of smooth maps or smooth section of a bundle. We will
    also use an overflow type  result that depends on our nonstandard model being
    polysaturated. See \cite{StrLux76} chapter 7; below we
    paraphrase their theorem 7.6.2 for our use.

\begin{theorem}\label{thm:extend external map}
    Suppose that $A\subset\rz\bbr^m$, not necessarily internal with cardinality less than that of $\rz\bbr^n$.
    Suppose that $F:A\ra\dstrcmn$ is any map. Then there is an
    internal subset $\bar{A}\subset\rz\bbr^m$ and an internal map
    $\bar{F}:\bar{A}\ra\dstrcmn$ such that $A\subset \bar{A}$ and
    $\bar{F}|_{A}=F$.
\end{theorem}

\subsection{Jet bundle constructions}
   In this section we cover enough of the basics of jets and the jet
   bundle formulation of (linear) differential operators, sufficient to
   formulate and prove our results.

\subsubsection{Jet bundle setup}
  We will briefly summarize that part of jet theory that we need.
  Although the following formulation is straightforwardly generalized
  to smooth manifolds, for brevity's sake we will restrict to the
  Euclidean case.
   Let $\bsm{P_k(m,n)}$ denote the polynomial maps of order $k$ from $\bbr^m$ to $\bbr^n$.
  If $f\in\dcm$, $x\in\bbr^m$ and $k\in\bbn$, let $\bsm{T^k_xf}\in P_k(m,n)$
  denote the $k$th order Taylor polynomial of $f$ at $x$. By *
  transfer, if $\rz P_k(m,n)$ denotes the $\rz\bbr$ vector space of internal polynomials
   from $\rz\bbr^m$ to $\rz\bbr^n$, $f\in\dstrcmn$ and
   $x\in\rz\bbr^m$, we have $\rz T^k_xf$, the internal $k^{th}$
   order Taylor polynomial of $f$ at $x$. Note that transfer implies
   that this has all of the properties of the Taylor polynomial,
   suitably interpreted. Note that although $f=\rz g$ so that $\xi\mapsto\rz T^k_\xi f$ is simply the transfer of the standard map $x\mapsto T^k_xf$, for
   $\xi\in\rz\bbr^m_{\infty}$ or $k\in\rz\bbn_{\infty}$, $\rz T^k_\xi f$ can be
   very pathological.

     We define an
  equivalence relation on \cmn\; as follows. We say that $f\in\dcmn$
  \textbf{vanishes to $k$th order at $x$} if $T^k_xf=0$ and for $f,g\in\dcmn$,
  we say that f equals g to $k$th order, written
   $\bsm{f\overset{x_k}\sim g}$ if $T^k_x(f-g)=0$. This defines an equivalence relation on \cmn.
  Let $\bsm{j^k_xf}$ denote the equivalence class containing $f$. We denote
  the set of equivalence classes by $\bsm{\SJ^k_{m,n,x}}$. There are a
  variety of definitions of $\SJ^k_{m,n,x}$ and its not hard to
  show that one can identify $\SJ^k_{m,n,x}$ with the set of Taylor polynomials of order $k$ at $x$ of
  smooth maps $(\bbr^m,x)\ra\bbr^n$ and we can identify $j^k_xf$ with
   $T^k_xf$. (An equivalence class consists of all maps with a given $k^{th}$ order Taylor polynomial at $x$.) Let $\bsm{\SJ^k_{m,n}}=\cup_{x\in\bbr^m}\SJ^k_{m,n,x}$.
  $\SJ^k_{m,n}$ is a smooth fiber bundle, in fact, as our maps have range $\bbr^m$, a vector bundle, over
  $\bbr^m$ with fiber over $x\in\bbr^n$ given by $\SJ^k_{m,n,x}$. Let
  $\bsm{\pi^k_0}:\SJ^k_{m,n}\ra\bbr^m$ denote the bundle projection. Note
  also that if $l,k\in\bbn$ with $l>k$, then $\SJ^l_{m,n}$ is a bundle
  over $\SJ^k_{m,n}$; let $\bsm{\pi^l_k}$ denote the bundle projection. We
  let $\bsm{C^{\infty}(\SJ^k_{m,n})}$ denote the $\bbr$ vector space of $C^{\infty}$
  sections of $\SJ^k_{m,n}$. If $f\in\dcmn$, there is a canonical
  section of $\pi^k_0$ given by $\bsm{j^kf}:x\mapsto j^k_xf$. There is a
  canonical map, the operation of taking the $k$ jet:
\begin{align}
   \bsm{j^k}:\dcmn \ra C^{\infty}(\SJ^k_{m,n})\quad f\mapsto j^kf
\end{align}

  For later purposes we also need to define the infinite jet,
  $\bsm{j^{\infty}_xf}$, for $f\in\dcmn$. Doing a simplified rendering of  projective
  limits, we will define the vector space of infinite jets at $x\in\bbr^m$, $\bsm{\SJ^{\infty}_{m,n,x}}$, to be the set of sequences
  $(f^0,f^1,f^2,\ldots)$  such that $f^k\in\SJ^k_{m,n,x}$ for all $k$ and for all nonnegative integers $j<k$, $\pi^k_j(f^k)=f^j$. Then for a given $f\in\dcmn$, the infinite jet of $f$ at $x$, $j^\infty_xf$, is clearly the well defined element of $\SJ^\infty_{m,n,x}$ given by $(j^0_xf,j^1_xf,j^2_xf,\ldots)$. It is easy to see that  $\SJ^{\infty}_{m,n,x}$, is an infinite dimensional
  vector space over $\bbr^m$, operations given componentwise, and
  that, for each $x\in\bbr^m$, the map $\bsm{j^{\infty}_x}:f\mapsto j^{\infty}_xf:\dcmn\ra\SJ^{\infty}_{m,n,x}$ is a
  vector space surjection with kernel the subspace of $g\in \dcmn$
  such that $j^k_xg=0$ for all integers $k$, ie., $g$ vanishes to
  infinite order at $x$. We will also need the forgetful fiber
  projection $\bsm{\pi_{k,x}}:\SJ^{\infty}_{m,n,x}\ra\SJ^k_{m,n,x}$. As the base  range space is
  linear, $\pi_{k,x}$ is a surjective linear morphism and clearly has kernel
  the (ideal) of formal power series that vanish to $k^{th}$ at $x$, see
  above .


   From the canonical (global) coordinate framing $\bsm{x_i}$,
  $1\leq i\leq m$ on $\bbr^m$, and $\bsm{y^j}$, $1\leq j\leq n$ on $\bbr^n$,
  we get induced coordinates, $\bsm{x_i, y^j_{\a}}$ for $|\a|\leq k$ and
  $1\leq j\leq n$, on $\SJ^k_{m,n}$ defined as follows. The $x_i$ are
  just the pullback of the coordinates on the base $\bbr^m$.  If
  $\phi\in\SJ^k_{m,n}$, then $\phi\in\SJ^k_{m,n,x_0}$ for some
  $x_0\in\bbr^m$ and so $\phi$ can be written as $j^k_{x_0}f$ for some
  $f\in\dcmn$. Then

\begin{align}
  x_i(\phi)=x_{0,i},\quad y^j_{\a}(\phi)\doteq\p^{\a}(f^j)(x_0)=\phi^j_{\a}.
\end{align}
   where $\bsm{\p^{\a}}$ denotes $\f{\p^{|\a|}}{{({\p}x_1)}^{\a_1}\cdots{({\p}x_m)}^{\a_m}}$
  where $\a=(\a_1,\ldots,\a_m)$.  If $\la\in
   C^{\infty}(\SJ^k_{m,n},\bbr^n)$, then with respect to the induced
  coordinates, we write this as $\bsm{\la(x_i,y^j_{\a})}$. therefore, for
  later use, we can Taylor expansion $\la$ \textbf{with respect to the $x_i$ coordinates}, around a given $p_0$, as follows. Let
  $p_0=(x_{0,i},y^k_{0,\a})$ be a coordinate representation as above.
  Let $\bsm{\la^l}$ denote the $l^{th}$ coordinate of $\la$ with respect to
  the canonical coordinates on $\bbr^n$. Then the Taylor expansion to
  order s with in the base coordinates is
\begin{align}\label{taylor exp 1}
   \la^l(x_i,y^k_{\a})=\sum_{|\a|\leq
   s}K_{\a}(x-x_0)^{\a}\p^{\a}(\la^l)(p_0)+ \tl{\la}^l(x_i,y^k_{\a})
\end{align}
   where for $|\a|=s+1$, $\tl{\la}^l\in
   C^{\infty}(\SJ^k_{m,n},\bbr^n)$ vanishes to order $s+1$ in the base coordinates at $p_0$ , $K_{\a}$ are the usual factorial
   constants, $(x-x_0)^{\a}=(x_1-x_{0,1})^{\a_1}\cdots (x_m-x_{0,m})^{\a_m}$, and
   $\bsm{\p^{\a}}$ is the ${\a}^{th}$ partial derivative with respect to the
   base coordinates. Therefore, for our purposes the vanishing order, normally defined
   in terms of the power of the maximal ideal at the given point in terms of the $x$
   coordinates, will be defined in terms of the degree of vanishing
   derivatives (in $x$ coordinates) at the given point. We need to emphasize that we are considering vanishing order of the smooth maps on the jet bundle only with respect to dependence on the base coordinates.

   Given the usual framing $\p_i=\f{\p}{\p x_i}$ , $1\leq i\leq m$ for
   $T\bbr^m$, $\p_{i,x}$ being the frame for $T_x\bbr^m$; we have an
   induced framing of $T\SJ^k_{m,n}$ given by adjoining to these tangent horizontal vectors the vectors  $\p_{y^j_{\a}}= \f{\p}{\p  y^j_\a}$ that are tangent to the fibers of the bundle projection $\pi^k_0$ at $x$, for $j=1,\ldots,n$ and $|\a|\leq k$.

   The notion of contact is useful in understanding the sharpening of the results here vis a vie the results of Todorov. Given $x\in\bbr^n$ and a nonnegative integer $s$, we say that $f\in\dcmn$ has \textbf{contact }$\bsm{s}$ with $\bbr^m$ at $x$ if $f(x)=0$ and the graph of $f$, $\G_f\subset\bbr^m\x\bbr^n$ is flat to $s^{th}$ order at $(x,0)$, that is, if $T^k_xf=0$, ie., if $j^s_xf$ is the equivalence containing the $0$ $s$ jet at $x$.
   We say that \textbf{$\bsm{f,g\in\dcmn}$ have $\bsm{s^{th}}$ order contact at $x$} if $f-g$ has $s^{th}$ order contact with $\bbr^m$ at $x$, the graph of the $0$ function, at $x$. It should be obvious that this is an equivalence relation and that the set of all $g\in\dcmn$ that belong to the $s^{th}$ order contact class of $f$ is precisely the affine subset with the same $s^{th}$ order jet as $f$.

\subsubsection{Prolonging jet maps and total derivatives}
   Let $\bsm{\bbr^p_m}$ denote the product bundle with fiber $\bbr^p$ and
   base $\bbr^m$; if $p=1$, we will denote this bundle by $\bsm{\bbr_m}$. If $x\in\bbr^m$, let
   $\bsm{\bbr^p_{m,x}}$ denote the vector space fiber of $\bbr^p_m$ over $x$.  In the following we will be using \textbf{vector bundle}
   \textbf{maps}, ie., smooth maps of bundles over the same base that preserve fibers and cover the identity map on the base. The symbols of linear differential operators,
   $\la:\SJ^r_{m,1}\ra\bbr^p_m$ are maps of this type. The set of such maps is a $C^{\infty}(\bbr^m,\bbr)$
    module and will be denoted by
    $\bsm{C^{\infty}(\SJ^r_{m,1},\bbr^p_m)}$.
   If  $\la: \SJ^k_{m,n}\ra \bbr^p_m$ is such a smooth bundle map, and $l\in\bbn$,
   then there exists a smooth bundle map $\bsm{\la^{(l)}}:\SJ^{k+l}_{m,n}\ra
   \SJ^l_{m,p}$\;, called the $\bsm{l^{th}}$\textbf{-prolongation of} $\bsm{\la}$ such that
   the following diagram is commutative
\begin{align}\label{diag1}
    \begin{CD}
                \SJ^{k+l}_{m,n}      @>\la^{(l)} >>     \SJ^l_{m,p}  \\
                @V\pi^{k+l}_k VV                         @V\pi^l_0VV\\
                \SJ^k_{m,n}          @>\la      >>
                \bbr^p_m
    \end{CD}
\end{align}
    That is, as $j^s_x(\la\circ j^rf)$ depends only on derivatives up to order $s$ of $y\mapsto j_y^rf$ at $x$,
    and so only on the $r+s$ jet of $f$ at $x$, then the following
    definition is well defined. If $f\in\dcmn$, then
\begin{align}\label{eqn:coord free prolng formula}
    \bsm{\la^{(s)}}(j^{r+s}_x(f))=j^s_x(\la\circ j^rf).
\end{align}

    The prolongation of vector fields on $\bbr^m$ to vector fields on
    $\SJ^k_{m,n}$ are given by fairly complicated recursion formulas.
   For treatments of prolongations of vector fields in somewhat
    different contexts, see Olver \cite{Olver1993}, p110 or Pommaret, \cite{Pommaret1978}, p253. Pommaret
    gives a remarkably easy derivation of these expressions. We only
    need the prolongation of coordinate vector fields, ie., total
   derivatives, and these have far simpler expressions. These give
   explicit local expressions of prolongations and so
    allow us to computationally investigate the
    effect of successive prolongations on maps of jets. For  each
    coordinate tangent field $\p_i$ on $T\bbr^m$, for $1\leq i\leq m$,
    we have an explicit expression for the corresponding lifted local
    section of $T\SJ^k_{m,n}$, the \textbf{total derivative} $\bsm{\p^\#_i}$ defined as
    follows.
\begin{align}\label{total der sum}
    \p_i^\#=
    \p_i+ \sum_{\substack{|\a|\leq k\\1\leq j\leq n}}
    y^j_{\a^i}\p_{y^j_{\a}}
\end{align}
   where $\a^i=(\a_1,\ldots,\a_i+1,\ldots,\a_n)$. Note that $\p_i^\#$
   depends on coordinates of order $k+1$, ie., for $\la\in
   C^{\infty}(\SJ^k_{m,n},\bbr^n_m)$, we have $\p_i^\#(\la)$ are the
   coordinates for a map $\la^{(1)}:\SJ^{k+1}_{m,n}\ra\SJ^1_{m,n}$ with respect to the induced jet coordinates. In fact, we
    have the following.

\begin{lemma}\label{lem:loc coord of 1-prolong}
   Suppose that $\la:\SJ^k_{m,n}\ra\bbr^n$ is a smooth map. Let $\la^j$
   for $j=1,\ldots,n$ denote the coordinates of $\la$ with respect to
   the standard coordinate basis for $\bbr^n$. Then the components of
   $\la^{(1)}$ with respect to the given coordinates are
   $\p_i^\#(\la^j_{\a}$)
\end{lemma}

\begin{proof}
   One can verify this  lemma and the expression (\ref{total der sum})
   using the local version of the definition for prolonging jet maps
  (\ref{eqn:coord free prolng formula}) when $s=1$, eg., see \cite{Olver1993}, p109; ie.,
\begin{align}\label{1 prolong coord eqn}
    \p_i^\#(\la)(j_x^{r+1}(f)=\p_i(\la\circ j^r f)(x)
\end{align}
    and then applying the chain rule to the right side of (\ref{1
    prolong coord eqn}).
\end{proof}

\subsection{Differential operators and their prolongations}

    In order to align with Todorov's setup we will now restrict the
    dimension of the range space to be $1$. The superscript $j$ enumerating range space
    components  will no longer appear.

    Let $\bsm{LPDO_r}$ denote the vector space of linear partial differential operators of degree less than or equal to $r$
   with  coefficients in \cm.  Suppose that $P\in LPDO_r$. Then there exists
   a smooth bundle map $\bsm{\la_P}: \SJ^r_{m,1}\ra\bbr_m$ called the \textbf{total symbol of $\bsm{P}$}
   such that if $f\in\dcm$, then  $P(f)=\la_P\circ j^rf$ as elements
   of \cm. If $\la:\SJ^r_{m,1}\ra\bbr_m$ is the symbol of an $r^{th}$
   order differential operator, $P$,  then the \textbf{$\bsm{s^{th}}$ prolongation}
   $\bsm{\la^{(s)}}:\SJ^{r+s}_{m,1}\ra\SJ^s_{m,1}$ is defined
   on $r+s$ jets above.
   As prolongations of differential operators mapping $\dcm$ to $\dcm$ are systems, we will use the notation $\bsm{LPDO_{r+s}^s}$ for $r+s$ order linear differential operators on $\dcm$ to smooth sections of $\SJ^s_{m,1}$. In particular, if $P\in LPDO_r$, and $s\in\bbn$, then there is
   $\bsm{P^{(s)}}\in LPDO_{r+s}^s$\;, called the
   $\bsm{s^{th}}$ \textbf{prolongation of} $\bsm{P}$, defined as the differential operator
   whose symbol is the $s^{th}$ prolongation of $\la_P$. We will have more to say about its nature later.

   On the nonstandard level, note that if $g\in\dcm$, then, at every standard $x$, ie.,
   $\rz x\in\;^\s\bbr^m$, and for each $k\in\;^\s\bbn_0=\bbn\cup\{0\}$,
   $\rz\!j_{*x}^k(\rz g)=\rz\!(j^k_xg)$. That is, the internal operator
   $\bsm{\rz\!j^k_{*x}}|_{\dcgcm}$ operating on standard functions is just the transfer of $j^k_x$. (At nonstandard points this is not true.) It
   therefore follows that $\rz\la\circ\rz j^k$, restricted to a standard
   map, $\rz g$ at a standard point $\rz x$, is just the *transfer of $\la\circ
   j^k_xg\in\bbr_{m,x}$. We shall give a sufficient account of the remainder of the nonstandard material needed after we recall a bit more standard geometry.

\section{Standard geometry: Prolongation and Vanishing}\label{section:prolong and vanishing}
   This section proves the standard results that allow the proof of
   Corollary \ref{cor: infinite order Todorov}.
   The idea here is desingularize our total symbol by ``lifting'' it to a sufficiently high jet level where we can then invoke
   a version of Todorov's result. In order to do this, we need some sort
   of correspondence between solutions of $P$ and those of $P^{(s)}$.
   We also need this procedure to decrease the vanishing order the coefficients
   of the $P^{(k)}$  as $k$ increases. Here is the idea. On the one hand, we can think of a symbol, $\la=\la_P$ of an LPDE as a smooth family of linear maps $x\mapsto\la_x$, and therefore one can think of vanishing order of $\la$ at $x_0$ as the ``flatness'' of the graph of this map at the point $x_0\in\bbr^m$. This relates directly to the Taylor polynomials of the smooth coefficients of $\la$ at $x_0$. On the other hand and more abstractly, there is a classic ``desingularization'' machinery for jet bundle map, eg., symbols of differential operators, that carries solutions to solutions, ie., prolongation. In this section we will relate the intuitive vanishing order to this prolongation method in order to get crude controls between jets in the domain and range of the prolongation of $\la$, in terms of these singularities.   This will be done in this section.

    We first need the proper notion of vanishing order of a linear bundle map
    $\la:\SJ^k_{m,1}\ra\bbr$.
\begin{definition}
   Let  $x_0\in\bbr^m$, and $c\in\bbn$. Then we say
   that $\bsm{\la}$\textbf{ vanishes to order (exactly)} $\bsm{c}$ \textbf{(in $\bsm{x}$) along $\bsm{(\pi^k_0)^{-1}(x_0)}$}, written $\bsm{x_0\in\SZ^c(\la)}$, if
   $\p^{\a}(\la)(p)=0$ for all $\a$ with $|\a|\leq c$ and for all $p\in (\pi^k_0)^{-1}(x_0)$.
    and, secondly, there exists some
   and $\beta$ with $|\beta|=c+1$ and $p_0\in (\pi^k_0)^{-1}(x_0)$ such that
   $\p^{\beta}(\la)(p_0)\not =0$. Note, as always, that $\p^{\a}$
   denotes the $\a^{th}$ derivative with respect to the $x_i$
   coordinates.
\end{definition}

   First note that this is more transparently stated as follows. Writing $\la=\sum_{\a}f_\a y_\a$ for some smooth $f_\a$'s, this condition is equivalent to stating that for each coefficient $f_\a$,  $\p^\beta f_\a (x_0)=0$ for all multiindices $\beta$ satisfying $|\beta|\leq k$;
   with the second condition being that there exists a coefficient $f_\a$ and a multiindex $\beta$ with $|\beta|=k+1$ such that $\p^\beta f_\a(x_0)\not=0$.
   Note also that although the notion of contact is closely related to vanishing order, we will not pursue this connection here.  In the next lemma, we don't need to restrict to jet mappings that
   are the symbols of elements of $LPDO_k$. When we consider how total
   derivatives change the vanishing order of jet maps, it will be
   essential to consider the particular form of jet maps that are
   symbols of elements of $LPDO_k$. Below we will be using the following notation. If $\beta=(\beta_1,\cdots,b_m)$ is a multiindex of order $k$; ie., $|\beta|=\beta_1+\cdots+\beta_m=k$ and $1\leq i_0\leq m$, then $\bsm{\beta_{i_0}}=(\beta_1,\ldots,\beta_{i_0}+1,\beta_{i_0+1},\ldots,\beta_m)$; so eg., $|\beta_{i_0}|=|\beta|+1$.

\begin{lemma}\label{lem1:decreasing order of 0}
    Let $\la\in C^{\infty}(\SJ^k_{m,1},\bbr_m)$  and suppose that
    $p_0\in\SJ^k_{m,1}$ is in $\SZ^c(\la)$. Then for some $i\in\{1,\ldots,m\}$, we have
    $p_0\in\SZ^{c-1}(\p_i(\la))$.
\end{lemma}

\begin{proof}
   By hypothesis, $\p^{\a}(\la)(p_0)=0$ for $|\a|\leq c$  and there is a multiindex $\beta$ with $|\beta|=c+1$ and an
    such that $\p^{\beta}(\la)(p_0)\not =0$. Write $\beta=\a_i$
   for some multiindex $\a$ with $|\a|=c$ and $i\in\{1,\ldots,m\}$.
   That is, $\p^{\a}(\p_i\la)(p_0)\not =0$. But
   if $\a$ is a multiindex such that $|\a|\leq c-1$, then $|\a_i|\leq
   c$ and so
    $\p^{\a}(\p_i\la)(p_0)=\p^{\a_i}(\la)(p_0)=0$ by hypothesis. But
   then by definition $p_0\in\SZ^{c-1}(\p_i\la)$, as we wanted to
   show .
\end{proof}

\subsection{Lifting solutions }

   At this point, it is important to note the explicit form of the jet
   maps that are symbols of elements of $LPDO_k$.

\begin{lemma}\label{lem: LPDO symbol}
   Suppose that $P\in LPDO_k$. Then $\la=\la_P:\SJ^k_{m,1}\ra\bbr_m$ can
   be written in local coordinates $(x_i,y_{\a})$ in the following
   form $\la= \sum_{|\a|\leq k
   }f_{\a}y_{\a}$ where $f_{\a}\in\dcm$.
\end{lemma}
\begin{proof}
   This is clear.
\end{proof}

   A remark is in order. Since for each $x\in\bbr^m$, $\la_P=\la_{P,x}:\SJ^r_{m,1,x}\bbr_{m,x}$ is linear, then we will consider look of the vanishing order of $x\mapsto\la_x$

\begin{lemma}
   Suppose that $P\in LPDO_r$ , $s\in \bbn$ and $f\in\dcm$ solves
   $P^{(s)}(f)=0$. Then $f$ solves $P(f)=0$.
\end{lemma}

\begin{proof}
   This is an easy unfolding of the definitions. Operationally,
   $P^{(s)}$ is defined on $f\in\dcm$ as follows.
   $P^{(s)}(f)=j^s\circ P(f)$. But then $P^{(s)}(f)=0$ implies that
   $j^s\circ P(f)=0$. But for $g\in\dcm$, $j^s(g)=0$ as a section of $\SJ^s_{m,1}$ if
   and only if $g$ is identically $0$, eg., letting $g=P(f)$, we get
   $P(f)=0$.
\end{proof}

\subsection{Lifting zero sets}

\subsubsection{ Prolongation of linear symbols}

\begin{remark}
   If $\la:\SJ^r_{m,1}\ra\bbr_m$ is the symbol of $P\in LPDO_r$, then
   w\textbf{e will instead  write $\SZ^c(\la)$ for $\pi^r(\SZ^c(\la))$. In particular,
   $\SZ^c(\la)$ will now be considered as a subset of the base space,}
   $\bbr^m$. Note that since $\la$ is linear on the fibers, with particular form as noted in
   Lemma \ref{lem: LPDO symbol}, then
   to say that $\bsm{x\mapsto\la_x}$ \textbf{vanishes to order $c$ at $x_0$ is the same as our fiber condition as this says that the coefficient $f_\a$ of our generic jet derivative $y_\a$ vanishes to order $c$ at $x_0$.}
   So, in the linear case, we have the following definition.
\end{remark}
\begin{definition}
   Suppose that $\la$ is the symbol of $P\in LPDO_r$. Let  $\bsm{\SZ^c(\la)}$ denote the $x\in\bbr^m$ where all of the coefficients of $\la$ vanish to order $c$ and at least one does not vanish to order $c+1$. If $\la$ is such a linear jet bundle map, let $\bsm{\SZ_\la}$ denote those $x\in\bbr^m$ where $\la_x$ is the zero linear map. Given this, it should be obvious that the
   conclusion of Lemma \ref{lem1:decreasing order of 0} holds with
   $\pi^r(\SZ^c(\la))$ in place of $\SZ^c(\la)$.
\end{definition}

With this development, we have the following initiating lemma.

\begin{lemma}\label{lem2:decreasing order of 0}
   Suppose that $P\in LPDO_r$ with $\la\in
   C^{\infty}(\SJ^r_{m,1},\bbr_m)$
   the symbol of $P$.
   Let $c\in \bbn$
   and $x_0\in\SZ^c(\la)$. Then $x_0\in\SZ^{c-1}(\la^{(1)})$. In particular, if
   $x_0\in\SZ^1(\la)$, then $\la^{(1)}_{x_0}\not=0$.
\end{lemma}
\begin{proof}
   So we have that $x_0$ satisfies $\p^{\a}\la(x_0)=0$ if $|\a|\leq
   c$ , and there exists multiindex $\beta$, with
   $|\beta|=c+1$ such that $\p^{\beta}\la(x_0)\not=0$.
   By Lemma (\ref{lem:loc coord of 1-prolong}) we only need to verify
   that for all $i,\a$ with $|\a|\leq c-1$,
   $\p^{\a}(\p_i^\#\la)(x_0)=0$ and there exists $i_0,\tl{\a}$ with
   $|\tl{\a}|=c$ such that
   $\p^{\tl{\a}}(\p_{i_0}^\#\la)(x_0)\not=0$.
   To this end, suppose that the following statement is valid.
   Let $d\in\bbn$ and suppose that for all $\a$ with $|\a|\leq d$, we have that
   that $\p^{\a}\la(p_0)=0$. Then for arbitrary given $i_0$
   and $\tl{\a}$ with $|\tl{\a}|=d$,
   $\p^{\tl{\a}}(\p^\#_{i_0}\la)(p_0)=0$ if and only if
   $\p^{\tl{\a}}(\p_{i_0}\la)(p_0)=0$. Then one can see that from this statement and Lemma
   \ref{lem1:decreasing order of 0}  the proof of our lemma follows immediately; so it suffices to prove the above statement.
   First of all, we need only to
   verify this statement for $\la=\la_P$, for $P=\sum_{|\a|\leq
   r}f_{\a}\p^{\a}$ where $f_{\a}\in\dcm$ for all $\a$; that is, if
   $\la=\sum_{|\a|\leq r}f_{\a}y_{\a}$. In the following calculations we will leave out evaluation at $x_0$; it is implicit. So

\begin{align}
   \p^\#_i(\la)=(\p_i+ \sum_{|\a|\leq r}y_{\a_i}\p_{y_{\a}})(\sum_{|\beta|\leq r}f_{\beta}y_{\beta})\notag   \\
   =\sum_{|\beta|\leq r}\p_i(f_{\beta})y_{\beta}+\sum_{|\a|\leq
   r}f_{\a}y_{\a_i}.
\end{align}
   That is
\begin{align}
   \p^\#_i(\la)=\sum_{|\a|\leq
   r}((\p_if_{\a})y_{\a}+f_{\a}y_{\a_i}).
\end{align}

   so that taking $\p^{\tl{\a}}$, for $|\tl{\a}|\leq d$, of both sides and interchanging
   derivatives, we get

\begin{align}\label{total der expression}
   \p^\#_i(\p^{\tl{\a}}\la)=\sum_{|\beta|\leq
   r}\p_i(\p^{\tl{\a}}f_{\beta})y_{\beta}+\p^{\tl{\a}}(f_{\beta})y_{\beta_i}.
\end{align}

   But by hypothesis, the second term, on the right side of (\ref{total der expression})
   is $0$ and the truth of the statement follows.
\end{proof}

\subsubsection{Higher prolongations; coordinate calculations}
   We want to prove a higher order prolongation version of the previous lemma. We need some preliminaries before we state the lemma.
   Suppose that $\a$ is a multiindex, let $\p_\a^\#$ denote the $\a^{th}$ iteration of the coordinate total derivatives; ie., if $\a=(\a_1,\ldots,\a_m)$ is a multiindex of order $k=\a_1+\cdots+\a_m$, then $\p^\#_\a=(\p^\#_1)^{\a_1}\circ\cdots\circ (\p^\#_m)^{\a_m}$, it being understood that if $\a_j=0$, then the corresponding $j^{th}$ factor is missing. As defined, $\p^\#_\a$ sends functions on $\SJ^r_{m,1}$ to functions on $\SJ^{r+k}_{m,1}$.
   Next, note that if $\la=\sum_\a f_\a y_\a$ is the symbol of $P\in LPDO_r$, then $\la^{(k)}$  is the symbol of the $r+k^{th}$ order operator $P^{(s)}$. As such, using the induced  coordinates $y_\a$, $|\a|\leq k$ on the range, $\SJ^k_{m,1}$, we can write $\la^{(k)}_x$ as $((\p_\a^\#\la)_x)_{|\a|\leq k}$. That is, $\la^{(k)}$ is given by the family of linear maps
\begin{align}
   x\mapsto ((\p^\#_\a\la)_x)_\a\;:\SJ^{r+k}_{m,1,x}\ra\SJ^k_{m,1,x}
\end{align}
   Note then that this family of maps vanishes to order $c$ at $x_0$ precisely when for each $\a$, $|\a|\leq r+k$ the component $x\mapsto\p^\#_\a(\la)_x$ vanishes to order $c$.
   Given this, we have the following extension of the previous lemma to general prolongations.
\begin{lemma}
   Suppose that the bundle map $\la:\SJ^r_{m,1}\ra\bbr_m$ is the symbol of $P\in LPDO_r$ and suppose that $x_0\in \SZ^c(\la)$. Then $\la^{(c+1)}_{x_0}$ is a nonzero linear map.
\end{lemma}
\begin{proof}
   By the remarks before the lemma, it suffices to prove that $\la^{(c+1)}_{x_0}$ has a nonzero component at $x_0$; that is, for some $\g$ with $|\g|\leq c+1$, $\p^\#_\g(\la)_{x_0}$ is a nonzero linear map. Note also that if $\a,\beta$ are distinct multiindices, then $\p^\#_\a$
   First note that if $g\in\dcm$ and  $\a,\g$ are appropriate multiindices, then
\begin{align}
   \p^\#_\g(gy_\a)=\sum_{\e+\rho=\g}\p_\e g\cdot y_{\a+\rho}
\end{align}
   Suppose that $g$ vanishes to order (exactly) at $x_0$; so that there is an index $i_0$ and multiindex $\bar{\beta}$ with $|\bar{\beta}|=c$, such that $\p_\a g(x_0)=0$ for $|\a|\leq c$ and $\p_{\bar{\beta}_{i_0}}g(x_0)\not= 0$. Then $\p^\#_{\bar{\beta}_{i_0}}(gy_\a)=\p_{\bar{\beta}_{i_0}}g\cdot y_\a$.
   This follows upon inspection: $\beta_{i_0}$ is the only multiindex of length $c+1$ occurring in the sum. All other multiindices occurring are of length $\leq c$ and so these terms are zero by the hypotheses on $g$.

   So now consider a general symbol $\la=\sum_{|\a|\leq r}f_\a y_\a$ and suppose, by hypothesis that $x_0\in \SZ^c(\la)$. So there exists a multiindex $\a_0$ with $|\a_0|\leq r$, a multiindex $\bar{\beta}$ of order $c$ and an index  $i_0\in\{1,\ldots,m\}$ such that $\p_{\bar{\beta}_{i_0}}f_{\a_0}(x_0)\not=0$. We will show that the $\bar{\beta}_{i_0}$ component of $\la^{(c+1)}_{x_0}$ is nonzero; ie., that $\p^\#_{\bar{\beta}_{i_0}}(\la)_{x_0}$ is a nonzero linear map. Now
\begin{align}
   \p^\#_{\bar{\beta}_{i_0}}(\la)_{x_0}=\sum_{|\a|\leq r}\p^\#_{\bar{\beta}_{i_0}}(f_\a y_\a)_{x_0}\qquad\quad\qquad\qquad \notag \\
   \qquad\qquad =\sum_{|\a|\leq r}\sum_{\g\leq\bar{\beta}_{i_0}}\p_\g(f_\a)(x_0)(y_{\a+(\bar{\beta}_{i_0}-\g)})_{x_0} \notag \\
   \qquad =\sum_{|\a|\leq r}\sum_{\substack{\g\leq\bar{\beta}_{i_0}\\|\g|=c+1}}\p_\g(f_\a)(x_0)(y_{\a+(\bar{\beta}_{i_0}-\g)})_{x_0}\notag \\
   =\sum_{|\a|\leq r}\p_{\bar{\beta}_{i_0}}f_\a(x_0)y_\a|_{x_0}. \quad\quad \qquad\qquad
\end{align}
   But the linear forms $y_\a|_{x_0}$ are linearly independent and the coefficient of $y_{\a_0}|_{x_0}$ is nonzero by hypothesis, hence the conclusion follows.
\end{proof}

\begin{corollary}\label{cor:removing finite zeroes}
  Suppose that $P\in LPDO_r$ with $\la\in
   C^{\infty}(\SJ^r_{m,1},\bbr_m)$
   the symbol of $P$ . Suppose that $c_0=\sup\{c\in\bbn:\SZ^c_{\la}\;\text{is nonempty}\}$
   is finite.
   Then  $\SZ(\la^{(c_0+1)})$ is empty. That is, for each $x\in\bbr^m$, $rank(\la^{(c_0)}_x)\geq 1$.
\end{corollary}

\begin{proof}
   This is an immediate consequence of the above lemma.

\end{proof}

\section{ Standard Geometry: Prolongation and rank}\label{section: standard jet work,prolong and rank}
   In section 3, we were concerned with the vanishing order of the (total) symbol of an element of $LPDO_r$. Our proofs involved calculations with the induced local coordinate formulation of prolongations of jet bundle maps. In this section, we will prove the standard results needed in the proof of Corollary \ref{cor: infinite order Todorov}. Here, we are a bit more tradionally concerned with the prolongation effects of regularity hypotheses on the principal symbol of an element of $LPDO_r$.
   Our constructions will instead be in the tradition of diagram chasing through commutative diagrams of jet bundles.

   Here we will show that if the principal symbol
   $\un{\la}:\SJ^r_{m,1,x}\ra\bbr_m$ is nonvanishing, then for each
   $k\in\bbn$, $\la^{(k)}:\SJ^{r+k}_{m,1,x}\ra\SJ^k_{m,1,x}$ is
   maximal rank. We will use this fact in constructing solutions for
   $P(f)=g$. We will first look at a coordinate argument that seems
   to indicate this. We will follow this with a simple version with a proof of this fact
   using a  typical jet bundle argument.

   Suppose that $\la:\SJ^r_{m,1}\ra\bbr_m$ is the symbol of $P\in
   LPDO_r$. Let's begin by looking at first order prolongations. Suppose that $x_0\in \bbr^m$ and that $\la_{x_0}:\SJ^r_{m,1,x_0}\ra\bbr_{x_0}$
   is nonzero in the sense that some coefficient $a_{\bar{\a}}$, for $|\bar{\a}|=r$ is nonzero at $x_0$.  . Then we will verify the
   first prolongation of $\la$,
   $\la^{(1)}_{x_0}:\SJ^{r+1}_{m,1,x_0}\ra\SJ^1_{m,1,x_0}$ has ``top order part'' of
   maximal rank.
    So we need to show that the rank of the ``top order part'' of
   $\la^{(1)}_{x_0}:\SJ^{r+1}_{m,1,x_0}\ra\SJ^1_{m,1,x_0}$ is $m$ as this
   is the dimension of the fiber of $\SJ^1_{m,1}$. Write
   $\la_{x}=\sum_{|\a|\leq r}a_{\a}(x)y_{\a}$ and, in coordinates,

\begin{align}
\la^{(1)}_{x_0}=(\la_{x_0},\p^{\#}_1\la_{x_0},\ldots,\p^{\#}_m\la_{x_0})
\end{align}
   where, as before  for each $i$,

\begin{align}\label{eqn:sum formula;tot der}
    \p^{\#}_i\la_{x_0}=\sum_{|\a|\leq r}(\p_i
    a_{\a}(x_0)y_{\a}+a_{\a}(x_0)y_{\a_i}).
\end{align}
   where by assumption $a_{\tl{\a}}(x_0)\not=0$ for
   some $\tl{\a}$ with $|\tl{\a}|=r$. But then the linear forms
   $a_{\tl{\a}}(x_0)y_{\tl{\a}_i}$ for $i=1,\ldots,m$ are linearly
   independent. Therefore, the linear forms $\p^{\#}_i\la_{x_0}$, by
   their above expressions (\ref{eqn:sum formula;tot der}), are also
   linear independent. Hence, their image spans the fiber of
   $\SJ^1_{m,1}\ra\bbr^m$ over $x_0$.

   Given that $\la$ is the (total) symbol of $P\in LPDO_r$; the ``top order part'' of $\la$, $\un{\la}$ is the principal symbol of $P$. We need a little more jet stuff to properly define the principal symbol and proceed with the general statement for arbitrary prolongations.
   The set of $j^{k+1}_xf$ of $f\in\dcm$ that vanish to $k^{th}$ order at $x$ have a
    canonical $\bbr$ vector space identification with $\bsm{\mathcal S_{k+1} T^*_x}$, the ${k+1}^{st}$ symmetric power of $T^*_x$,
    the cotangent space to $\bbr^m$ at $x$. (See Pommaret, \cite{Pommaret1978}, p47, 48.) In fact, for every $x\in\bbr^m$ and $k\in\bbn$ we have a canonical injection of vector spaces,
    $\mathcal S_{k}T^*_x\stackrel{i_{k}}{\ra}\SJ^{k}_{m,1,x}$ which is a canonical isomorphism
    when $k=1$, ie., $T^*_x\cong \SJ^1_{m,1,x}$. This injection embeds in an exact sequence:
 \begin{align}
      0\ra\mathcal S_{k+1}T^*_x\stackrel{i_{k+1}}{\ra}\SJ^{k+1}_{m,1,x}\stackrel{\pi^{k+1}_k}{\ra}\SJ^k_{m,1,x}\ra 0.
 \end{align}

   In fact this and much of the following holds in far greater generality; eg., giving exact sequences of
   jets of bundles over any paracompact smooth manifold. Next, when $k=r$ the order of our operator,
   note that expression (\ref{eqn:sum formula;tot der}) when
   evaluated on  $j^{r+1}_xf\in\mathcal S_{k+1}T^*_x$ gives
 \begin{align}\label{eqn: principal symbol sum}
      \p^\#_i\la_{x}(j^{k+1}_xf)=\sum_{|\a|\leq
      r}a_\a(x)y_{\a_i}(j^{k+1}_xf),
 \end{align}
   the other terms being zero as $j^k_xf=0$.
   Now for each $k=0,1,2,\ldots$, the \textbf{principal symbol of} $\bsm{P^{(k)}}$, denoted
   $\bsm{\un{\la}^{(k)}}:\mathcal S_{r+k}T^*_x\ra\mathcal S_{k}T^*_x$ is the $\bbr$
   linear map induced by the restriction to $\mathcal S_kT^*_x$ of
   $\la^{(k)}$. As the principal symbol
   $\un{\la}:\mathcal S_{r}T^*_x\ra\bbr_{m,x}$ can be represented by $\sum_{|\a|=r}a_\a y_\a$,
    then $\un{\la}^{(1)}:\mathcal S_{r+1}\ra T^*_x$ can be written $\sum_i\sum_{|\a|=r}a_{\a}
    y_{\a_i}\otimes dx_i$. See the discussion below.
    Also note that the linear maps $y_\a|_{\mathcal S_{k}T^*_x}$ decomposes as the
   $\a$ symmetric product of the coordinate cotangent vectors and  we
   have the canonical $\bbr$ vector space identification $\mathcal S_{r+1}T^*_x\cong
   \mathcal S_r T^*_x\otimes T^*_x$. With these preliminaries we can prove
   the following.
\begin{lemma}
   Suppose that for  $x\in\bbr^m$,
   $\un{\la}_x:\SJ^r_{m,1,x}\ra\bbr_{m,x}$ is nonzero, ie. a surjection. Then
   $\la^{(k)}$ is a surjection for all $k\in\bbn$.
\end{lemma}
\begin{proof}
   The remark above allows us to decompose  the expression for
   $\un{\la}^{(1)}$ as $1_{T^*_x}\otimes\un{\la}$. Actually this
   holds at all levels.(SEE POMMARRET, \cite{Pommaret1978} p193) That is,
  \begin{align}
     \un{\la}_x^{(k)}=\un{\la}_x\otimes 1|_{\mathcal S_kT^*_x}.
  \end{align}
   But note that in the category of finite dimensional $\bbr$ vector
   spaces, the tensor product of surjections is a surjection, hence
   by hypothesis $\un{\la}^{(k)}$ is a surjection for each
   $k\in\bbn$.

   We will prove, by induction on  $k$, the order of prolongation,
   that $\la^{(k)}$ is a surjection. The result holds for $k=0$.
   Suppose that it holds for some $k\geq 0$. We will prove that it
   holds for $k+1$.
    By the inductive hypothesis, the remark on $\un{\la}^{(k)}$ directly above and general facts on jets,  we have a commutative
   diagram of exact sequences of linear maps over $x$\\
\begin{align}\label{diag2}
    \begin{CD}
                    0         @.                         0\\
                   @VVV                                    @VVV\\
                 \mathcal S_{r+k+1}T^*_x          @> \un{\la}^{(k+1)} >>\mathcal S_{k+1}T^*_x  @>>> 0\\
                 @VV\i_{r+k+1}V                            @VV\i_{k+1}V\\
                 \SJ^{r+k+1}_{m,1,x}      @>\la^{(k+1)} >>  \SJ^{k+1}_{m,1,x}    \\
                @VV\pi^{r+k+1}_{r+k} V                               @VV\pi^{k+1}_{k} V\\
                \SJ^{r+k}_{m,1,x}          @>\la^{(k)}      >>
                \SJ^k_{m,1,x}    @>>> 0\\
                @VVV                                    @VVV\\
                    0            @.                     0\\
\end{CD}
\end{align}
   and we wish to prove that the middle row is a surjection. Suppose
   that $\eta_{k+1}\in\SJ^{k+1}_{m,1,x}$. We will find
   $\z\in\SJ^{r+k+1}_{m,1,x}$ such that $\la^{(k+1)}(\z)=\eta_{k+1}$.
   The proof will be a typical 'diagram chase'. Let
   $\eta_k=\pi^{k+1}_k(\eta_{k+1})$. Then, by hypothesis, there
   exist $\eta_{r+k}\in\SJ^k_{m,1,x}$ such that
   $\la^{(k)}(\eta_{r+k})=\eta_k$. Let $\eta_{r+k+1}\in
   (\pi^{r+k+1}_{r+k})^{-1}(\eta_{r+k})$. So  commutativity of
   the lower square implies that
  \begin{align}
     \pi^{k+1}_k\circ\la^{(k+1)}(\eta_{r+k+1})=\la^{(k)}\circ\pi^{r+k+1}_{r+k}(\eta_{r+k+1})=\eta_{k}= \pi^{k+1}_{k}(\eta_{k+1}).
  \end{align}
   That is, 1) $\pi^{k+1}_k(\la^{(k+1)}(\eta_{r+k+1})-\eta_{k+1})=0$, and so by exactness of the right sequence, there is
   $\s_{k+1}\in\mathcal S_{k+1}T^*_x$ such that
  \begin{align}\label{eqn:comm diagram,1}
     i_{k+1}(\s_{k+1})=\la^{(k+1)}(\eta_{r+k+1})-\eta_{k+1}
  \end{align}

    But $\un{\la}^{(k+1)}$ is
   surjective, ie., there exists $\s_{r+k+1}\in\mathcal S_{r+k+1}T*_x$
   such that \\ $\un{\la}^{(k+1)}(\s_{r+k+1})=\s_{k+1}$.
   So by this and commutativity of the top square, we have
  \begin{align}\label{eqn:comm diagram,2}
     \la^{(k+1)}\circ i_{r+k+1}(\s_{r+k+1})=i_{k+1}\circ
     \un{\la}^{(k+1)}(\s_{r+k+1})=i_{k+1}(\s_{k+1}).
  \end{align}
   Combining the equivalences in (\ref{eqn:comm diagram,1}) and (\ref{eqn:comm
   diagram,2}), we get
  \begin{align}
     \la^{(k+1)}(\eta_{r+k+1}-i_{r+k+1}(\s_{r+k+1}))=\eta_{k+1}.
  \end{align}
   That is, $\z\doteq\eta_{r+k+1}-i_{r+k+1}(\s_{r+k+1})$ is the
   element of $\SJ^{r+k+1}_{m,1,x}$ we are looking for.
\end{proof}
   Note that we did not use the full strength of our setting; we did not use
   the exactness of the left vertical sequence.

\begin{corollary}\label{lem:symb prolong is surj}
   Suppose that $\la:\SJ^r_{m,1}\ra\bbr_m$ is such that $\un{\la}_{x_0}:\SJ^r_{m,1,x}\ra\bbr_{x_0}$ is nonzero as in the previous
   lemma. Then if $g\in\dcm$, $j^k_{x_0}g\in Im(\la^{(k)}_{x_0})$, for
   every $k\in\bbn_0$.
\end{corollary}
\begin{proof}
   This is an immediate consequence of the previous lemma.
\end{proof}
   This corollary will allow us to extend Todorov's pointwise equality  to an infinite jetwise equality in the next section.
\section{The Main Linear Theorem}
   Before we begin the transfer of our results to the nonstandard world, we need in place a bit more of the framework for the infinite jet results in this section. First, let $N_s$ denote the finite set of
   multiindices of length $m$ and weight less than or equal to $s$\;;
   ie., indexing the fiber jet coordinates for $\SJ^r_{m,1}$. Let
   $\ov{N_s}\subset N_s$ denote the subset of multiindices of length
   equal to $\a$.
   Given the notational material, we now examine how the lifting works.
   Todorov proved a result crudely stated as follows: given $g$, there exists nonstandard $f$ such that $P(f)(x)=g(x)$ at each standard $x$. Our intention is to prove that such $f$ exists such that $j^s_x(P(f))=j^s_xg$ for all standard $x$ and all $s\in\;^\s\bbn$. This will be a consequence of the material in the previous section, the transfer of the Borel lemma and a bit more standard preliminaries.
    The mapping $\la^{(s)}$ can be seen as the intermediary of $j^s(P(f))$ as follows. If $s\in\bbn$, and $|\a|\leq s$, we have that
\begin{align}
   j^s_{x_0}(P(f))=P^{(s)}(f)(x_0)=\la^{(s)}_{x_0}(j^{r+s}_{x_0}(f))=j^s_{x_0}(\la\circ j^rf)=j^s_{x_0}g.
\end{align}
   We can therefore get a good estimate on the size of the range the successive prolongations of the range of $P$ at $x_0$ by watching the mapping properties of $\la^{(s)}$.


    We will denote by  $\bsm{\la^{(\infty)}}$ the infinite prolongation of
    $\la$ given by
\begin{align}
   j^{r,\infty}_xf\mapsto j^{\infty}_x(\la\circ j^rf): \SJ^{(r,\infty)}_{m,1,x}\ra\SJ^{(\infty)}_{m,1,x}
\end{align}
     where $j^{r,\infty}_xf=(j^r_xf,j^{r+1}_xf,j^{r+2}_xf,\ldots)$,
     ie., $\la^{(\infty)}$ being the map whose components are
     already defined.

    In this section and the next the transfer of the Borel Lemma will be used. Here is a statement of the version we will use.
\begin{lemma}[Borel Lemma]\label{lem: borel}
   Let $x\in\bbr^m$ and suppose that $\phi\in\SJ^\infty_{m,1,x}$. Then there exists $f\in\dcm$ such that $\phi = j^\infty_xf$.
\end{lemma}
   Note, implicit in this result is the fact that this determination depends only on the germ of $f$ at $x$.

\subsection{Transfer of jet preliminaries}
   To prove the main theorem we need to transfer the above jet formulation to
   the internal arena inserting the homogeneous version of Todorov's
   result into a jet level high enough so that the symbol has the
   correct form.

   If $\bsm{{}^\s LPDO_r}$ denotes those elements $P$ of $\rz
   LPDO_r$ whose coefficients are standard elements of \cm , then
   these correspond to symbols $\la_P\in{}^\s C^{\infty}(\SJ^r_{m,1},\bbr)$. Therefore, a special case of the
   *transfer of Corollary \ref{cor:removing finite zeroes} is the following statement.
\begin{corollary}
   Let $r\in\ {}^\s\bbn$, $D_a=\{x\in\bbr^m:|x|\leq a\}$  and  $\rz P\in {}^\s LPDO_r$ with $\la$   the symbol of $P$.
   Suppose that $\max\{c\in {}^\s\bbn:\rz\SZ^c({\la})\cap \rz
   D_a\;\text{is nonempty}\}$ is bounded in ${}^\s\bbn$
   independent of $a\in\bbn$. Then there exists $s\in {}^\s \bbn$ such that if
   $\la'$ is the symbol of $P^{(s)}$, then $\SZ_{\la'}\cap\rz\bbr^m_{nes}$ is empty.
\end{corollary}
\begin{proof}
   In *transferring Corollary  \ref{cor:removing finite zeroes}, we
   need only note the following things for this corollary to follow.
   First of all, we *transfer this corollary, for the situation
   where $\SZ(\la_P)\subset D_a$ for a given $0<a\in\bbn$ noting that
   $\cup_{a>0}\rz  D_a=\rz\bbr^m_{nes}$ and that the hypothesis
   implies that there exists $a_0\in\bbn$ such that $m(a)\doteq\max\{c\in {}^\s\bbn:\rz\SZ^c({\la})\cap \rz
   D_a\;\text{is nonempty}\}$ satisfies $m(a)\leq m(a_0)$ for all
   $a\in\bbn$.
\end{proof}
\begin{remark}
   Suppose that
   $\la_P\in C^{\infty}(\SJ^r_{m,1},\bbr)$ is such that for every bounded
   $B\subset\bbr^m$,
   $\SZ(\la_P)\cap B$ has no accumulation points. Then $\rz\la_P$ can't have  the
   property that $\la_P$ vanishes to infinite, but hyperfinite order
   at some point in $\rz\bbr^m_{nes}$. It therefore follows that
   we can't use *transfer to generalize this result, in the given context, to points where
   $ \la_P$ vanishes to infinite, hyperfinite order.

   In order to proceed we need a particular type of nonstandard partition of unity
   construction. For $0<c\in\rz\bbr^m$ and $y\in\rz\bbr^m$, let
   $D_c(y)$ denote the disk centered at $y$ with radius $c$.
\end{remark}
\begin{lemma}[*Weak partition of unity]\label{lem:POU}
   Suppose that for every $ x\in
   \bbr^m$, we have $f^x\in \dstrcmn$. Then there exists $f\in\dstrcmn$ and $0<\delta\sim 0$
   such that for each $x\in\bbr^m$, $f|_{D_\delta(x)}=f^x|_{D_\delta(x)}$.
\end{lemma}
\begin{proof}
   First of all, sufficient saturation implies that the (external) map $^\s\bbr^m\ra\dstrcmn:x\mapsto f^x$ extends
   to an internal map $\SI:\rz\bbr^m\ra\dstrcm:l\in\SI\mapsto f^l$; see Theorem \ref{thm:extend external map}. Let
   $\SL\subset\rz\bbr^m$ be a *finite subset such that
   $^\s\bbr^m\subset\SL$. Choose $0<\delta\in\rz\bbr$ such that $\delta<\f{1}{10}\rz min\{|l-l'|:l,l'\in\SL,l\not=l'\}$. By the *transfer of a variation on a weak form of the
   partition of unity construction, there exists $\psi_l\in\dstrcm$
   for each $l\in\SL$ such that $\sum_{l\in\SL}\psi_l(x)=x$ for each
   $x\in\rz\bbr^m$ and for each $l\in\SL$, $\psi_l|_{D_\delta(l)}\equiv
   1$. (As the *cardinality of $\SL$ is *finite, we don't have to worry about *local finiteness of the sum of the $\psi_l$'s.)
    Then the function $f\doteq\rz\sum_{l\in\SL}\psi_lf^l$ has the
    properties we need.
\end{proof}
\begin{remark}
   In a follow up paper, a numerically controlled version of this lemma (and the corresponding one in the nonlinear section) will allow proof of most of the existence results in this paper within the category of Colombeau-Todorov algebras.
\end{remark}

   Let $\la:\SJ^r_{m,1}\ra\SJ^0_{m,1}$ denote the symbol of a $P\in
   LPDO_r$.
\begin{definition}
   Let $\bsm{finsupp(P)}$ or $\bsm{finsupp(\la)}$ denote the subset of  $\bbr^m$ given  by $\cup\{\SZ^c(\la):c=0,1,2,\ldots\}$
   For each $x\in\bbr^m$ and $k\in\bbn_0$, let $\bsm{\SJ^k_{\la,x}}$
   denote the subspace of $\SJ^k_{m,1}$ given by $\la^{(k)}(\SJ^{r+k}_{m,1,x})$.
   We write $g\not=0(\la,x)$, if $j^k_xg\not=0$ for some $k\in\bbn$. Let
   $\bsm{\SV^m_x}<\dcm$ denote the ideal of $f\in\dcm$ such that $j^{\infty}_xf=0$.
\end{definition}
\begin{lemma}\label{lem: infin diml soln space at x}
   If $x\in finsupp(\la)$, then there exists $g\in\dcm$ such that
   $j^{k_0}_xg\not=0$ for some $k_0\in\bbn$ and
   $j^k_xg\in\SJ^k_{\la,x}$ for every integer $k\geq 0$.
\end{lemma}
\begin{proof}
   Given Corollary \ref{cor:removing finite zeroes} the assertion amounts to specifying that the derivatives at each level must lie in a given set
   and hence is an easy consequence of the Borel lemma.
   \end{proof}
\begin{definition}
   Let $\bsm{\frak I_{\la,x}}=\{g\in\dcm:j^k_xg\in\SJ^k_{\la,x}\;\text{for
   all}\;k\in\bbn_0\}$.
\end{definition}
   Note, of course, that $\SV^m_x<\frak I_{\la,x}$.
   So by the above lemma, $\frak I_{\la,x}$ is infinite dimensional.
   Therefore, $\rz\frak I_{*\la,*x}$ is a *infinite dimensional $\rz\bbr$ subspace of \strcm.
   In the nonstandard world, we have the following analogous definition.
\begin{definition}
   Let
\begin{align}
    \pmb{^\s\frak I_{\la,x}}= \{g\in\dstrcm: \rz j^k_{*x}g\in\rz\SJ^k_{*\la,*x}
   \;\text{for all}\;k\in\; ^\s\bbn_0\}.
\end{align}
\end{definition}
   Note that $^\s\frak I_{\la,x}$ is an external $\rz\bbr$ vector space and $\rz\frak I_{\la,x}\subset\;^\s\frak
   I_{\la,x}\subset\dstrcm$. In particular, $^\s\frak I_{\la,x}$ is infinite dimensional.
    Note that its *dimensionality is not well defined. We have one more definition.
\begin{definition}
   Let $\bsm{^\s\frak I_\la}$ denote the set of $g\in\dstrcm$ such that for
   all $\rz x\in\;^\s\bbr^m$, $g\in\;^\s\frak I_{\la,x}$.
\end{definition}

\begin{lemma}\label{lem: finite vanish implies section}
   Suppose that $^\s\frak I_{*\la,*x}\not=0$ for some $x\in\bbr^m$.
   Then $^\s\frak I_\la\not=0$.
\end{lemma}
\begin{proof}
   For each $x\in\bbr^m$, choose $f^x\in\dstrcm$ with $f^x\in\; ^\s\frak
   I_{*\la,*x}$, such that for some $x$, $f^x\not=0(\la,x)$.
   By Lemma \ref{lem:POU}, there exists $f\in\dstrcm$ and
   $0<\delta\sim0$ such that $f|_{D_\delta(x)}=f^x|_{D_\delta(x)}$ for each
   $x\in\bbr^m$. But then, for each $x\in\bbr^m$ and each
   $k\in\bbn_0$, $\rz\!j^k_{*x}f=\rz\!j^k_{*x}f^x\in\; ^\s\SJ^k_{*\la,*x}$. That
   is $f\in ^\s\frak I_\la$, and $f\not=0(\la,x)$ for some $x$.
\end{proof}
\begin{lemma}\label{lem:infty soln jet at x}
   Let $x\in\bbr^m$, and  $g\in\frak I_{\la,x}$. Then
   there exists $f\in\dcm$ such that
   $\la^{(\infty)}_x(j^{r,\infty}_xf)=j^{\infty}_xg$.
\end{lemma}
\begin{proof}
   First of all, for every $k\in\bbn_0$, there exists
   $\g_k\in\SJ^{r+k}_{m,1}$ such that $\la^{(k)}(\g_k)=j^k_xg$. This
   just follows from the definition of $\frak I_{\la,x}$. Since this
   holds for all $k$, then there exists $\g\in\SJ^{r,\infty}_{m,1}$
   with $\la^{(\infty)}_x(\g)=j^{\infty}_xg$. Just let $\g$ be such
   that $\pi^{\infty}_k(\g)=\g_k$ for each $k$. But note that for
   $\g\in\SJ^{r,\infty}_{m,1,x}$, the Borel Lemma, Lemma \ref{lem: borel}, implies that there exists $f\in\dcm$ such that
   $j^{r,\infty}_xf=\g$.
\end{proof}
\begin{notation}
   If $f\in\dstrcm$, we will denote
     \begin{align}
        \bsm{\rz\!j^\s_x(f)}=(\rz(j^k_{*x})(f))_{k\in\bbn_0},\;\text{an external sequence}.
     \end{align}
   Similarly, if $\la$ is an internal jet map and we are considering, for each $k\in\bbn$, not $\rz\bbn$, $\la^{(k)}_{\;* x}$,  the \textbf{internal} prolongation of $\la$ at the standard point $\rz x$, ie.,$(\la^{(k)}_{*x})_{k\in^\s\bbn}$, then we will also write this as $\bsm{\la^{(\s)}_x}$; eg., if $\la$ or $f$ are standard and we are considering only this family of internal prolongations of $\rz\la$ or $\rz f$, then we will write $\bsm{\rz\!j^\s_x(\rz\!f)}$ or $\bsm{\rz\!\la^{(\s)}_x}$.
\end{notation}
   In the situation when $f\in\dcm$,  $\rz\!j^\s_x(\rz f)$ is just the external sequence
   of standard numbers, $(\rz(j^k_xf))_{k\in\bbn_0}$. This notation
   can be unwieldy; some of the parentheses, or *'s may be left out
   if the meaning is still clear.
   Note that if
\begin{align}
   \SV^m=\bigcap_{x\in\bbr^m}\SV^m_x=\{f\in\dstrcm:\rz\!j^\s_x(f)=0,\;\text{for
   all}\;x\in\bbr^m\},
\end{align}
    then $\SV^m<\;^\s\frak I_\la$. Although $\SV$ is a $\rz \bbr$ vector space, it is nonetheless external. To get a sense of the size of
    $\SV^m$ in \strcMn, note that $\SL<\SV$ where $\SL$ is the
    *finite codimensional subspace \strcMn\; defined in the
    concluding section of the paper. Therefore, we have the following consequence of Lemma \ref{lem: finite vanish implies section}.
\begin{corollary}\label{cor: section at pt implies infnte diml sections}
   Suppose that $^\s\frak I_{*\la,*x}\not=0$ for some $x\in\bbr^m$.
   Then $^\s\frak I_\la$ is *infinite dimensional.
\end{corollary}
\begin{remark}
   Suppose that $f\in\dcm,\;\ov{f}\in\dstrcm$ such that for some standard $x$,
   and $0<\delta\sim 0$, $\ov{f}|_{D_\delta(*x)}=\rz\!f|_{D_\delta(*x)}$. Then the
   internal jet sequence
   $\rz j^{\infty}_{*x}\ov{f}\doteq(\rz j^k_{*x}\ov{f})_{k\in*\bbn}$ is just the *transfer of the standard sequence $(j^k_xf)_{k\in\bbn_0}$,
    eg.,  when the set of jet indices is  restricted to to the external set
    $^\s\bbn_0$. That is, in the above notation, $\rz\!j^\s_x(\ov{f})=
    (\rz\!j^{k}_xf)_{k\in\bbn_0}$.
\end{remark}

\subsection{Many generalized solutions with high contact}
   The following result is the main linear result of the paper, although it's import is not apparent without the following corollaries.

\begin{theorem}\label{thm:lin eqn,infinite contact soln}
   Suppose that $P\in LPDO_r$. Then, for every $g\in\; ^\s\frak
   I_\la$, there exists $f\in\dstrcm$ such that
\begin{align}
    \rz\!j^{\s}_x(\rz P(f))=\rz\!j^{\s}_xg\;\text{for
    every}\;x\in\bbr^m.
\end{align}
   That is, $\rz P(f)$ has $^\s$infinite order *contact with $g$ at all points of $\rz\bbr^m_{nes}$.
\end{theorem}
\begin{proof}
   Suppose that $^\s\frak I_{\la}$. By Lemma \ref{lem:infty soln jet at x} if $x\in\bbr^m$, there is
   $f^x\in\dcm $ such that
\begin{align}\label{eqn:infinite jet soln at point}
   \la^{(\infty)}_{P,x}(j^{\infty}_xf^x)=j^{\infty}_xg.
\end{align}
    By Lemma
   \ref{lem:POU} there exist $\ov{f}\in\dstrcm$ such that for every
   $x\in\bbr^m$ $\ov{f}|_{D_\delta(*x)}=\rz f^x|_{D_\delta(*x)}$. By the remark
   above, for each such standard $x$,
   $\rz\!j^{\s}_x\ov{f}=\rz\!j^{\s}_{*x}(\rz\!f^x)$. But this
   implies that, at each standard $x$,
\begin{align}
   \rz\!\la^{(\s)}_{*x}(\rz\!j^{\s}_{*x}\ov{f})=\rz\!\la^{(\s)}_{*x}(\rz\!j^{\s}_{*x}\rz\!f^x)
\end{align}
   Coupling this with the transfer of expression (\ref{eqn:infinite jet soln at
   point}) restricted to standard indices, we now have $\ov{f}\in\dstrcm$ such that
\begin{align}\label{eqn:infinite prolong,global soln in symb}
   \rz\la^{(\s)}_{*x}(\rz j^{\s}_{*x}\ov{f})=\rz j^{\s}_xg
\end{align}
   for each $x\in\bbr^m$. But, by definition of prolongation,
   *transferred
\begin{align}\label{eqn:infin prolong P is infin prolong lam}
   \rz j^{\infty}_{*x}(\rz P(\ov{f}))=\rz\la^{(\infty)}_{*x}(\rz j^{r,\infty}_{*x}\ov{f}).
\end{align}
   Stringing together expressions (\ref{eqn:infinite prolong,global soln in
   symb}) and (\ref{eqn:infin prolong P is infin prolong lam}), restricted to standard indices, gets
   our result, as this holds for every standard $x$.
\end{proof}
\begin{corollary}\label{cor: infinite order Todorov}
   Suppose that $P\in LPDO_r$ with symbol $\la$, and principal symbol $\un{\la}$. Suppose that for each $x\in\bbr^m,\un{\la}_x\not=0$. Then for every
   $g\in\dstrcm$, there exists $f\in\dstrcm$ with
\begin{align}
    *j^{\infty}_{*x}(*P(f))=*j^{\infty}_{*x}g\;\text{for
    every}\;x\in\bbr^m.
\end{align}
\end{corollary}
\begin{proof}
    By Lemma \ref{lem:symb prolong is surj}, if $g\in\dcm$, then $g\in\; ^\s\frak
    I_\la$. But then the result is a direct consequence the previous
    theorem.
\end{proof}
\begin{remark}
   To put this result in perspective, note that Todorov, \cite{Todorov96},
    proves the $0^{th}$ order jet case in his paper, with a slightly
    weaker hypothesis.
\end{remark}
    For those $x\in\bbr^m$ where $\la_x=0$, a trivial case for the
    $0$-jet, as Todorov notes, becomes a nontrivial thickened result
    when the consideration becomes the infinite jet at standard points
    where some finite prolongation $\la^{(k)}_x$ is nonzero. For
    this situation we have the following result.
\begin{corollary}\label{cor: infinite order solns for finite contact}
    Suppose that $finsupp(P)=\bbr^m$. Then there exists an *infinite dimensional subspace $^\s\frak I_P<\dstrcm$
    such that if $g\in\; ^\s\frak I_P$,
     then there exists $f\in\dstrcm$ such that
\begin{align}
    \rz j^{\s}_{*x}(\rz P(f))=\rz j^{\s}_{*x}g\;\text{for
    every}\;x\in\bbr^m.
\end{align}
\end{corollary}
\begin{proof}
   As $finsupp(P)=\bbr^m$, if $x\in\bbr^m$, Lemma \ref{lem: infin diml soln space at
   x} implies that $\mathcal S^P_x\doteq\{j^{\infty}_xg:g\in \frak I_{\la,x}\}$ is nonzero. Therefore, the result follows from Corollary \ref{cor: section at pt implies infnte diml sections} and the above
   theorem.
\end{proof}
     That is, even if the symbol vanishes at points of $\bbr^m$, as long as this vanishing order is finite at each such point, then there exists many $g\in\dstrcm$, satisfying the above compatibility conditions, such that $\rz P(f)=g$ is solved to infinite order along $^\s\bbr^m$ by $f\in\dstrcm$.
\subsection{Solutions for singular Lewy operator}
   Before we move to the next section, let's look at the Lewy operator, see \cite{Todorov96}, p.679,  $\SL=\p_1+i\p_2-2i(x_1+ix_2)\p_3$ acting on smooth complex valued functions on $\bbr^3$. First of all, note that the results just proved hold just as well with complex valued functions; the proofs are identical. Second, note that the principal symbol, $\un{\la}_{\;\SL}$ of $\SL$ is the same as the total symbol $\la_\SL=y_1+iy_2-2i(x_1+ix_2)y_3$. Inspection shows that these maps are nonvanishing, hence $\SL$ satisfies the hypotheses in Corollary \ref{cor: infinite order Todorov}, ie., for any $g\in \rz C^\infty(\bbr^3,\bbc)$, there exists (many) $f\in\rz C^\infty (\bbr^3,\bbc)$ such that $\rz\SL(f)(\rz x)=\rz g(\rz x)$ to infinite order at each $x\in\bbr^3$.
   But we can say more. Suppose that $h=(h_1,h_2,h_3)$ is such that $h_i\in C^\infty(\bbr^3,\bbr)$ for each $i$ and $h$ vanishes to finite order at each $x\in\bbr^3$. {\it Let $\wh{\SL}=h_1(x)\p_1+ih_2(x)\p_2-2ih_3(x)(x_1+x_2)\p_3$, a kind of singular Lewy operator with finite singularities at each $x\in\bbr^3$.
   Then Corollary \ref{cor: infinite order solns for finite contact} implies that for any $g\in C^\infty(\bbr^3,\bbc)$ that vanishes where $\la_{\wh{\SL}}$ vanishes to order at least that of $h$, there exists $f\in \rz C^\infty(\bbr^3,\bbc)$ such that $\wh{\SL}(f)(\rz x)=g(\rz x)$ holds to infinite order at all $x\in\bbr^3$}.

\section{Nonlinear PDE's and the pointwise lifting property}\label{section: nonlinear work}
   In this section $P$ can now be an arbitrary smooth nonlinear PDO of finite order.
   Only the rudiments of a nonlinear development parallel to the linear considerations
   in the previous sections will be attempted in this paper. The point here is
   that the
   framework is not an impediment to a consistent consideration of generalized objects.

   First, as it is natural within our framework, we straightforwardly extend the notion of solution of a differential equation, as defined in Todorov's paper, to include nonlinear as well as linear differential equations.
   In analogy with $LPDO_r$, a (possibly nonlinear) order $r$ partial differential operator,  $P:\dcm\ra \dcm$, is a mapping given by $P(f)(x)=\la(j^r_xf)$ where now the total symbol of $P$, $\la:\SJ^r_{m,1}\ra\bbr_m$ is a possibly nonlinear smooth bundle map. Let $\bsm{NLDO_r}$ denote this set of operators.
\begin{definition}
    Given $g\in\dstrcm$, we say that $f\in\dstrcm$ is a solution of $\rz P(f)=g$ if $\rz P(f)(\rz x)=g(\rz x)$ for every $x\in\bbr^m$.
\end{definition}
   We will consider a simple set theoretic condition on pairs $(P,g)$ (or $(\la_P,g)$) the \textbf{pointwise covering property}, $\bsm{PCP}$.
   An easy (saturation) proof will get that if $(P,g)$ satisfies this
   property, written $(P,g)\in PCP$, or $(\la_P,g)\in PCP$, then $P(f)=g$ has generalized solutions in a sense of
   Todorov. We will then show that the main theorem is a corollary
   of this result by verifying that our linear differential equation
   satisfies $PCP$.
\begin{definition}
   Let $\la\in C^{\infty}(\SJ^k_{m,1},\bbr)$ and $g\in\dcm$. We
   say that \textbf{the pair $\bsm{(\la,g)}$ satisfies $\bsm{PCP}$}, if for each
   $x\in\bbr^m$, there exists $p\in (\pi^k)^{-1}(x)$, such that
   $\la(p)=g(x)$. If $\la\in \rz C^{\infty}(\SJ^k_{m,1},\bbr)$ and
   $g\in\dstrcm$, then we say that \textbf{the pair $\bsm{(\rz\la,g)}$ satisfies $\bsm{{}^\s PCP}$}
   if for all $x\in {}^\s\bbr^m$, there exists $p\in\rz\pi^{-1}(x)$,
   such that $\la(p)=g(\rz x)$.
\end{definition}

\begin{remark}
   Note that finding $p\in(\pi^r)^{-1}(x)$ such that $\la(p)=g(x)$
   is identical to finding $h\in\dcm$ such that $h$ solves  $P(h)=g$
   at the single point $x$. Also, note the relationship between $PCP$ and ${}^\s PCP$.
   If $\la\in C^{\infty}(\SJ^k_{m,1},\bbr)$ and $g\in\dcm$ are
   such that $(\la,g)\in PCP$, then $(\rz\la,\rz g)\in {}^\s PCP$. On the other hand, if
   $\la\in SC^{\infty}(\SJ^k_{m,1},\bbr)$ and $g\in\dscm$ are
   such that $(\la,g)\in {}^\s PCP$, then $(^o\la,^og)\in PCP$. (Recall that if $X,Y$ are Hausdorff topological spaces and $f:\rz X\ra\rz Y$ is such that $f$ maps nearstandard points of $\rz X$ to those of $\rz Y$, then the standard part of $f$, $^of:X\ra Y$ is a welldefined map.)
\end{remark}

   The following lemma verifies that the the PCP condition restricted to linear differential operators has Todorov's criterion as a special case. We are working with the symbol of the operator.
\begin{lemma}
   Let $P\in LPDO_r$ and write
$$
   \la_P=\sum_{|\a|\leq r}f_{\a}y_{\a}.
$$
   Suppose that $\sum_{|\a|\leq r}|f^{\a}(x)|\not=0$ for all
   $x\in\bbr^m$
   Then for all $g\in\dstrcm$, $(\rz\la_P,g)\in {}^\s PCP$.
\end{lemma}

\begin{proof}
    Let $x_0\in\bbr^m$. We will not write $\rz x_0$ when we transfer. The condition guarantees that there exists a
    multiindex $\a$ such that $f^{\a}(x_0)\not=0$. Let $\G=\{\a:c_{\a}\doteq f^{\a}(x_0)\not=0\}$.
    If $\G$ has only one element, $\a_0$, let $h\in\dstrcm$ be such
    that
\begin{align}
    \rz\p^{\a_0}(h)(x_0)=\f{g(x_0)}{\rz f^{\a_0}(x_0)}
\end{align}
    Then if $\kappa\in\rz\SJ^k_{m,1}$ is given by $\rz j^r_{x_0}h$, we get that

\begin{align}\label{eqn:PCP}
    \rz\la_P(\kappa)=\rz\la_P(\rz j^r_{x_0}(h))= \\ \rz \sum_{|\a|\leq   \notag
    r}f^{\a}(x_0)y_{\a}(j^r_{x_0}h)= &
    f^{\a_0}(x_0)y_{\a_0}(j^r_{x_0}h)\notag = \\
    f^{\a_0}(x_0)\f{g(x_0)}{f^{\a_0}(x_0)}= g(x_0)\notag
\end{align}
    as we wanted. So suppose that $\G$ has at least two elements.
    Let $\a_0\in\G$ and let $\La=\G-\{\a_0\}$. Choose $h\in\dstrcm$ so
    that if $\a\in\La$, then $\rz\p^{\a}h(x_0)=0$ and (as in the first
    case)
    such that $\rz\p^{\a_0}(h)(x_0)=\f{g(x_0)}{f^{\a_0}(x_0)}$. Then as
    in expressions (\ref{eqn:PCP}), we get $P(h)(x_0)=g(x_0)$.
\end{proof}

    Given the above lemma we shall see that Todorov's result is a  corollary of this lemma and the next theorem proving the existence of solutions of PCP operators.
    Before we proceed to the theorem, we need some NSA preliminaries.
    First we give a simple example
    of the construction we will need.  Let $F(\bbr)$ be all maps from $\bbr$ to
    $\bbr$ and let $F^{\infty}(\bbr)=\{f\in F(\bbr):f\;\text{is smooth}\}$. Let $f\in F(\bbr)$, then
    there exists an (internal) element $\tl{f}\in\rz F^{\infty}(\bbr)$ such that
    $\tl{f}|_{{}^\s\bbr}=f$, as the following argument shows. Let
     $\SY_1\subset \rz\bbr$ be  *finite such that ${}^\s\bbr\subset\SY_1$ and let
     $\SY_2=\rz f(\SY_1)$. Then $\SY_2$ is obviously a *finite subset of $\rz\bbr$.
     Now consider the following elementary standard statement. If $S_1,S_2$ are finite subsets of $\bbr$, of the same cardinality,
     and $h:S_1\ra S_2$ is a bijection, there exists $\tl{h}\in F^{\infty}(\bbr)$ such that $\tl{h}|_{S_1}=h$. This follows from a simple partition of unity argument.
     Now *transfer this to get existence of $\tl{f}\in\rz F^{\infty}(\bbr)$ such that
     $\tl{f}|_{\SY_1}=\rz f|_{\SY_1}$. In particular, $\tl{f}|_{{}^\s\bbr}=\rz
     f|_{{}^\s\bbr}=f|_{\bbr}$, as we wanted. Now we want to do the
     same construction in the venue of bundles and their sections.
     Let $\bsm{\G(\SJ^r_{ m,1})}=\{s:\bbr^m\ra\SJ^r_{m,1}|\;\pi^r\circ
    s=\bbi_{\bbr^m}\}$, ie., set theoretic sections of $\pi^r$. Let
    $\bsm{\G^{\infty}(\SJ^r_{m,1})}=\{s\in\G(\SJ^r_{ m,1}):s\;\text{is a
    smooth map}\}$. We have the following lemma.
\begin{lemma}\label{lem:first standard jet approx}
    Suppose that $s\in\G(\SJ^r_{m,1})$. Then there exists
    $\tl{s}\in\rz\G^{\infty}(\SJ^r_{m,1})$, such that
    $\tl{s}|_{{}^\s\bbr^m}=s|_{\bbr^m}$.
\end{lemma}
\begin{proof}
    As with the above example, let $\SX\subset\rz\bbr^m$ be *finite
    such that $\bbr^m\subset\SX$. We have the following elementary
    fact. If $B=\{b_1,\ldots,b_l\}$ is a finite subset of the
    base and $P=\{p_1,\ldots,p_l\}\subset\SJ^r_{m,1}$ is a finite
    subset such that $p_j\in(\pi^r)^{-1}(b_j)$ for each $j$, then
    there exists $s\in\G(\SJ^r_{m,1})$ such that $s(x_j)=p_j$ for
    all $j$. Now *transnfer this statement, applying the *transferred
    statement to the *finite subset $\SX$ in the base and the
    *finite subset $\rz s(\SX)$ of points in the *bundle over $\SX$.
    That is, we can infer the existence of an internal section
    $\tl{s}\in\rz\G^{\infty}(\SJ^r_{m,1})$ such that for all $x\in\SX$,
    $\tl{s}(x)=\rz s(x)$, in particular $\tl{s}|_{{}^\s\bbr}=s$, as
    we wanted.
\end{proof}
    In the context of this lemma, we have that
    $(\rz\la,g)\in{}^\s PCP$ is equivalent to the existence of a set
    theoretic section $s\in\rz\G(\SJ^r_{m,1})$ such that the pointwise
    condition $\rz\la\circ \rz s=g$ holds on ${}^\s\bbr^m$. It's
    important to note that, generally speaking, such sections are
    far from integrable; that is, equal to $j^rf$ for some smooth
    $f\in\dcm$. But again, by a transfer argument, we can find such a
    section.

\begin{lemma}\label{lem:second standard jet approx}
    Suppose that $s\in\G^{\infty}(\SJ^r_{m,1})$ and $\SX\subset\rz\bbr^m$.
    Then there exists $f\in\dstrcm$ such that $\rz
    j^rf|_{\SX}=s|_{\SX}$.
\end{lemma}
\begin{proof}
    This just follows from the *transfer of the following obvious
    standard statement about jets. If
    $\{p_1,p_2,\ldots,p_l\}\subset\SJ^r_{m,1}$ such that
    $x_j=\pi^r(p_j)$ are all distinct. Then there exists $f\in\dcm$
    such that $j^r_{x_i}f=p_i$ for all i.
\end{proof}
    With these preliminaries, the proof of the following result is
    immediate.

\begin{theorem}
   Let $\SD\in NLDO_r$ and let
   $\la_{\SD}=\la\in C^{\infty}(\SJ^r_{m,1},\bbr)$ and $g\in\dstrcm$.
   Suppose that $(\rz\la,g)\in {}^\s PCP$. Then $\SD(f)=g$ has a generalized
   solution, $f$, in the sense of Todorov.
\end{theorem}

\begin{proof}
    By the remark above,  $(\rz\la,g)\in {}^\s PCP$ is equivalent to
    the existence of an $s\in\rz\G(\SJ^r_{m,1},\bbr)$ such that for
    every $x\in\bbr^m$, $\la_{\rz\SD}(s(\rz x))=g(\rz x)$. But by
    Lemma \ref{lem:first standard jet approx}, there exists $\tl{s}\in\rz\G^{\infty}(\SJ^r_{m,1})$
    such that $\tl{s}(\rz x)=s(x)$ for all $x\in \bbr^m$. And by
    Lemma \ref{lem:second standard jet approx}, there exists
    $f\in\dstrcm$, such that for all $x\in\bbr^m$, $\rz
    j^r_{* x}(f)=\tl{s}(\rz x)$.
\end{proof}
    Todorov's existence result (being for linear operators only) is a special consequence of the
    previous development.
\begin{corollary}
    Suppose that $g\in\dstrcm$ and $P\in LPDO_r$ is such that $\la_P$ is nonvanishing
    on $\bbr^m$. Then there exists $f\in\dstrcm$, such that for all
    $x\in\bbr^m$, $P(f)(\rz x)=g(\rz x)$, ie., $f$ is a solution of
    $P(f)=g$ in the manner of Todorov.
\end{corollary}
\begin{proof}
    This is clear.
\end{proof}
   Given the nonlinear setting of this section, proving results analogous to those in the linear sections appear to need much more involved preliminaries and so will be pursued at a later date.
   Nonetheless, it seems clear  that we can consider some general criteria
   revolving around when $(P,g)\in PCP$. In particular, it appears that we can
   prove {\it a universal existence theorem asserting that any possible space of
   generalized functions that has the  $PCP$ property is already
   contained in our nonstandard space}. This, too, will appear as time allows.

\section{Conclusion}
\subsection{Too many solutions?}
   In this paper I have used some of the machinery of the geometry of partial differential equations
   to explore the possibilities of the approach of Todorov. (We have yet to work through the nonlinear analogs of
   the linear results presented here; this will entail a much more extensive use of the the jet theory
   of nonlinear partial differential operators. Note even more starkly than in this paper; no counterpart in standard mathematics exists.)
   The implications of the results of this paper are still not clear. Yet one
   thing should be obvious, the class of internally smooth maps are
   remarkably `flabby', as compared to the standard world.

   As an indication of this, we have the following construction.
   let $\bsm{\SL}<\dstrcmn$ be the $\rz\bbr$ linear subspace of $\dstrcmn$ defined as follows.
   Let $\SY\subset\rz\bbr^m$ be a *finite subset such that
    $^\s\bbr^m\subset\SY$. Let $\omega\in\rz\bbn_{\infty}$. Then,
     the set  $\bsm{\SL}=\{f\in\dstrcmn:\rz\! j^{\om}_x f=0\;\text{for all}\;x\in\SY\}$ is a
    *cofinite dimensional subspace of \strcmn,  as this set of conditions
    on elements $f$ in \strcmn\; given by specifying the value of $j^{\om}_{x}f$, a *finite number of *Taylor coefficients at a *finite set of points in $\rz\bbr^m$, ie., the points of $\SY$ is *finite.
    Now, by construction,
   $\SL\cap\;^\s C^\infty(\bbr^m,\bbr^n)=\{0\}$, and
    we have the following diagram

\begin{align}
    \begin{CD}
         \SL\\
         @VjVV \\
         \dstrcmn  @>\text{ *$j^\om$}>> \rz \SJ^{\om}_{m,n}    @>\rho>>    \rz \SJ^{\om}_{m,n}|_{^\s\!\bbr^m}\\
         @AiAA \\
         \rz\bbr\otimes\;^\s C^\infty(\bbr^m,\bbr^n)
    \end{CD}
\end{align}
    where the maps $i$ and $j$ are $\rz\bbr$ subspace injections and $\rho$ is
    the highly external restriction to the fibers over $^\s\bbr^m$.
    Let $\Phi =\rho\circ\rz j^{\om}$. Then the following holds.
\begin{lemma}
   $\Phi|_{Im(j)}$ has image $\{0\}$ and $\Phi|_{Im(i)}$ is an injection.
\end{lemma}
\begin{proof}
    By construction, we have $\Phi(f)=0$ for every $f\in\SL$. On the other hand, if for any element
    $f\in\; \dcgcmn$ we have $\rz\!j^{\om}_x(\rz f)=0$ for each
    $x\in\SY$, we have in particular that $j^{\infty}_x f=0$, for
    each $x\in\bbr^m$, that is $f=0$. This therefore holds for all
    $f\in\rz\bbr^m\underset{\bbr^m}{\otimes}\dcgcmn$.
\end{proof}
    So we have that the subspace of elements of\; \strcmn \; whose $^\s$\! infinite *jet vanishes everywhere
    on $^\s\bbr^m$ is all of \;\strcmn\; up to a *finite dimensional subspace containing all
    standard smooth maps. It should therefore be clear that we have the immediate corollary that exemplifies the ability to
    bend almost all of \strcmn \; away from contact with the world
    of  standard differential equations, at least at standard points.
\begin{corollary}
    If $P\in NPDO_r$ for any $r\in\bbn$ such that $P(\text{zero map})= \text{zero map}$, then $\rz P(f)(\rz x)=0$ for all $f\in\SL$ and all $x\in\bbr^m$.
\end{corollary}
\begin{proof}
    All classical differential operators $P$ of order $r$, factor as $\rz P=\rz\la_P\circ \rz j^r=\rz\la_P\circ\rz\pi^{\om}_r\circ
    \rz j^{\om}$ and by above $j^{\om}(\SL)|\;^\s\bbr^m =\{0\}$.
\end{proof}
    {\it That is, all classical partial differential operators sending the zero map to the zero map operate as zero maps on ``almost all'' of \;\strcmn\;.}

    One perspective on the results here should not be a surprise: that *smooth functions (and with some thought *analytic functions) are far too flabby on a full infinitesimal scale. From a positive viewpoint, one could see how this might allow an investigator to have wide latitude in `Tayloring' generalized functions (on the monadic level) to get appropriate rigidities-growth or to test various empirical results by infinitesimal adjustings of singular parameters. The  remark (in the introduction) with respect to the work of Baty, etal, see eg., \cite{BatyShockWave2008} seems relevant to the second perspective. The algebras of Oberguggenberger and Todorov, \cite{OT98} and the further developments in eg., Todorov and Vernaeve, \cite{TodorVern2008} seems to be good examples of the Tayloring capacities.

\subsection{Prospects and goals}
   Only the rudiments of jets on the one hand, and nonstandard
   analysis, on the other have been deployed in this paper.
   In follow up articles we intend to use (*transferred) tools from
   smooth function theory along with a more extensive use of jet theory to extend
   both the linear and nonlinear existence results.  Further, deploying more
   nuanced version of the jet material of section \ref{section: standard jet work,prolong and
   rank} over certain types of infinite points in the jet fibers, we
   intend to prove results on regularity of solutions of partial
   differential operators, linear or nonlinear, whose symbols satisfy certain properness
   conditions. Our first paper along this line, \cite{McGafRegSolnNPDE}, gives a regularity theorem for a broad class of nonlinear differential operators.
   We also intend to extend the results here to
   include the results of Akiyama, \cite{AkiyamaNSSolvabilityOpsOnVBs} into the framework
   established here in the manner we have included the results of Todorov.
   The method is by an extension from internal mapping with *finite support
   to internal smooth modules of bundle sections with *finite support.
   Furthermore, as noted in the introduction, we will refine the arguments in this paper to Todorov's nonstandard Colombeau algebras.
   Given that all of the usual constructions on the symmetries of
   differential equations (as in eg., Olver, \cite{Olver1993}) are
   straightforwardly lifted to the nonstandard universe, we are also
   looking into developing a theoretic framework on generalized
   symmetries (eg., shock symmetries) of differential equations, continuing within  the jet
   theoretic framework begun here.



\bibliographystyle{amsplain}
\bibliography{nsabooks}

\end{document}